\documentclass[3p,times]{elsarticle}
\usepackage[colorlinks,linkcolor=blue,citecolor=blue]{hyperref}
\usepackage{bookmark}
\usepackage{bm}
\usepackage{amsmath,amsfonts,amsmath,amssymb,amsthm,cases}
\usepackage{algorithm,algorithmic}
\usepackage{multirow}
\usepackage{longtable}
\usepackage{fullpage}
\usepackage{cleveref}
\usepackage{graphics,graphicx,epsfig,epstopdf,subfigure,booktabs}
\usepackage{color,natbib}
\usepackage[dvipsnames, svgnames, x11names]{xcolor}  %Choices of colors, cf: latex_color.png
\usepackage{mathtools}

\newtheorem{theorem}{Theorem}[section]
\newtheorem{lemma}{Lemma}[section]

\newtheorem{remark}{Remark}[section]
\newtheorem{example}{Example}[section]
\numberwithin{equation}{section}

\newcommand{\bb}{\boldsymbol}

\begin{document}
\begin{frontmatter}
	\title{A lowest-order locking-free nonconforming virtual element method based on the reduced integration technique for linear elasticity problems}	
	\author{Yue Yu}
	\ead{terenceyuyue@sjtu.edu.cn}
	\address{School of Mathematical Sciences, Institute of Natural Sciences, MOE-LSC, Shanghai Jiao Tong University, Shanghai, 200240, P. R. China.}
	
	\begin{abstract}
      We develop a lowest-order nonconforming virtual element method for planar linear elasticity, which can be viewed as an extension of the idea in {\color{blue} Falk (1991)} to the virtual element method (VEM),  with the family of polygonal meshes satisfying a very general geometric assumption. The method is shown to be uniformly convergent for the nearly incompressible case with optimal rates of convergence. The crucial step is to establish the discrete Korn's inequality, yielding the coercivity of the discrete bilinear form. We also provide a unified locking-free scheme both for the conforming and nonconforming VEMs in the lowest order case. Numerical results validate the feasibility and effectiveness of the proposed numerical algorithms.
	\end{abstract}

	
\end{frontmatter}

\section{Introduction}
\label{Sec1}

The linear elasticity problem is of fundamental importance in elastic mechanics. The design and analysis of its numerical algorithms can help to solve or analyze more complex engineering problems.
The Lam\'{e} constant $\lambda$ characterizing the compressibility of the underlying materials poses a challenge for numerical computations. For small $\lambda$, this problem can be treated as the Poisson equation in vector form.
However, when it tends to infinity or the material is nearly incompressible, it becomes difficult to design a low order finite element methods (FEMs) with parameter-free or locking-free convergence.
For this reason, several methods or techniques have been developed in the literature to tackle with this problem.
One can refer to \cite{Brenner-Sung-1992,Hughes-Cohen-Haroun-1978,Malkus-Hughes-1978}
and the references therein for details.

The virtual element method (VEM) is a new numerical method proposed in recent years
as a generalization of the finite element method on polygonal or polyhedral meshes
(cf. \cite{Ahmad-Alsaedi-Brezzi-Marini-2013,Beirao-Brezzi-Cangiani-Manzini-2013, Beirao-Brezzi-Marini-Russo-2014}).
Compared with the FEM, it may be much easier to construct a locking-free virtual element
due to the flexibility of the construction of the virtual element spaces.
The conforming VEMs for the linear elasticity problem are first proposed in \cite{Beirao-Brezzi-Marini-2013},
where a locking-free analysis is carried out for the virtual element spaces of order $k\ge 2$. The order requirement is to
ensure the so-called discrete inf-sup condition and the optimal convergence. For the lowest-order case, a Bernardi-Raugel type VEM can be found in \cite{Tang-Liu-Zhang-Feng-2020}, where the uniform convergence is achieved by adding extra degrees of freedom so that the inf-sup condition can be satisfied easily.
Nonconforming VEMs for the linear elasticity problems are first introduced in \cite{Zhang-Zhao-Yang-Chen-2019}
for the pure displacement/traction formulation in two or three dimensions.
The proposed method is robust with respect to the Lam\'{e} constant for $k\ge2$, which, however, may be unstable for $k=1$ since the discrete Korn's inequality fails in the lowest-order case. For this reason, the authors in \cite{Kwak-Park-2021} present two kinds of lowest-order VEMs with consistent convergence, in which the first one is achieved by introducing a special stabilization term to ensure the discrete Korn's inequality, and the second one can be seen as an extension of the
idea of Kouhia and Stenberg suggested in \cite{Kouhia-Stenberg-1995} to the virtual element method.
Some other VEMs for elasticity problems in two and three dimensions can be found in
\cite{Artioli-Miranda-Lovadina-Patruno-2017,Artioli-Miranda-Lovadina-Patruno-2018,Beirao-Lovadina-Mora-2015,
Berbatov-Lazarov-Jivkov-2021,Caceres-Gatica-Sequeira-2019,Dassi-Lovadina-Visinoni-2020,
Dhanush-Natarajan-2019,Gain-Talischi-Paulino-2014,Mora-Rivera-2020,
Reddy-Huyssteen-2019,Zhang-Feng-2018}.

In this paper, we intend to generalize the idea in \cite{Falk-1991} to construct a lowest-order nonconforming locking-free VEM for solving the linear elasticity problems. The technique is referred to as the reduced integration technique in the literature of FEMs (cf. \cite{Hughes-Cohen-Haroun-1978,Malkus-Hughes-1978,Brenner-Sung-1992}), which is shown to be effective to construct parameter-free methods for both the conforming and nonconforming finite elements.

We end this section by introducing some notations and symbols frequently used in this paper.  For a bounded Lipschitz domain $D$, the symbol $( \cdot , \cdot )_D$ denotes the $L^2$-inner product on $D$, $\|\cdot\|_{0,D}$ denotes the $L^2$-norm, and $|\cdot|_{s,D}$ is the $H^s(D)$-seminorm. For all integer $k\ge 0$, $\mathbb{P}_k(D)$ is the set of polynomials of degree $\le k$ on $D$. For the vector-valued functions or spaces, we use the bold symbols, such as $\bb{u}$, $\bb{v}$, $\bb{L}^2(\Omega)$, $\bb{H}^1(\Omega)$, etc.
Moreover, for any two quantities $a$ and $b$, ``$a\lesssim b$" indicates ``$a\le C b$" with the hidden constant $C$ independent of the mesh size, and ``$a\eqsim b$" abbreviates ``$a\lesssim b\lesssim a$".

\section{The pure traction problem of linear elasticity} \label{sec:traction}

Let $\Omega\subset \mathbb{R}^2$ be a convex polygonal domain. Given an external force $\bb{f}$, the pure traction problem of linear elasticity is to find the displacement field $\bb{u} = (u_1,u_2)^\intercal$ such that
\begin{equation}
\label{model}
\begin{cases}
- {\rm div}~\bb{\sigma}(\bb{u}) = \bb{f}  \quad & \text{in} ~~\Omega, \\
\bb{\sigma}(\bb{u})\bb{n} = \bb{0} \quad & \mbox{on} ~~\partial\Omega,
\end{cases}
\end{equation}
where $\bb{n}$ denotes the exterior unit vector normal to $\partial\Omega$. The constitutive relation for linear elasticity is
\[\bb{\sigma}(\bb{u}) = 2\mu\bb{\varepsilon}(\bb{u}) + \lambda({\rm div} \bb{u})\bb{I},\]
where $\bb{\sigma} = (\sigma_{ij})$ and $\bb{\varepsilon} = (\varepsilon_{ij})$ are the second order stress and strain tensors, respectively, satisfying $\varepsilon_{ij} = \frac{1}{2}(\partial_i u_j + \partial_j u_i)$, $\lambda$ and $\mu$ are the Lam\'{e} constants, $\bb{I}$ is the identity matrix, and ${\rm div}~\bb{u} = \partial_1 u_1 + \partial_2 u_2$. Note that the problem \eqref{model} is solvable if the following compatibility condition is satisfied:
\[\int_\Omega \bb{f}\cdot\bb{v}{\rm d}x = 0, \quad \bb{v} \in RM(\Omega):= \{ \bb{v} \in \bb{H}^1(\Omega): \bb{\varepsilon}(\bb{v}) = \bb{O} \}.\]

We introduce the space
\begin{equation}\label{Htilde}
\bb{V}={\widetilde {\bb{H}}^1}(\Omega) := \Big\{ \bb{v} \in \bb{H}^1(\Omega): \int_{\partial \Omega} \bb{v} \text{d}s = \bb{0},\quad \int_\Omega  {\rm rot}~\bb{v} \text{d}x = 0 \Big\},
\end{equation}
where ${\rm rot} ~\bb{v} = \nabla  \times \bb{v} = \partial_1v_2 - \partial_2v_1$, which is slightly different from the usually used space ${\widehat {\bb{H}}^1}(\Omega)$ given by (cf. \cite{BrennerScott2008})
\begin{equation}\label{Hhat}
\bb{V}={\widehat {\bb{H}}^1}(\Omega) := \Big\{ \bb{v} \in \bb{H}^1(\Omega): \int_{\Omega} \bb{v} \text{d}s = \bb{0},\quad \int_\Omega  {\rm rot}~\bb{v} \text{d}x = 0 \Big\},
\end{equation}
 with the first constraint replaced by an integrand on the whole domain $\Omega$. The choice here is considered for two reasons, one is to ensure the interpolation of the VEM function lies in the underlying virtual element space, and the other is to make the terms associated with Lagrange multipliers computable in the implementation. One can refer to Subsection \ref{subsec:implementation} for details.
For the pure traction problem, it is well-known that the following Korn's inequality
\begin{equation}\label{Korn2}
|\bb{v}|_1 \le C \|\bb{\varepsilon}(\bb{v})\|_0
\end{equation}
holds for all $\bb{v}\in {\widehat {\bb{H}}^1}(\Omega)$. One can prove that it is also valid for $\bb{v}\in {\widetilde {\bb{H}}^1}(\Omega)$ by checking the proof of Theorem 11.2.12 in \cite{BrennerScott2008}, where only the second constraint of \eqref{Htilde} is utilized, and the first constraint is just to remove the additive constant vector in $RM(\Omega)$.

Let $\bb{f}\in\bb{L}^2(\Omega)$. The variational formulation of \eqref{model} is to find $\bb{u}\in\bb{V}$ such that
\begin{equation}
\label{Elasticity_TensorEq_Var}
a(\bb{u}, \bb{v}) = (\bb{f}, \bb{v}), \quad \bb{v} \in \bb{V},
\end{equation}
where
\begin{equation}\label{bilinear}
a(\bb{u}, \bb{v})
= 2\mu (\bb{\varepsilon}(\bb{u}), \bb{\varepsilon}(\bb{v}))
+ \lambda ({\rm div}~\bb{u}, {\rm div}~\bb{v}),
\end{equation}
and
\[
(\bb{f}, \bb{v})
= \int_{\Omega} \bb{f} \bb{v} \mathrm{d}x .
\]
For ease of presentation, we introduce the following notations:
\[a_\mu(\bb{u}, \bb{v}) = (\bb{\varepsilon}(\bb{u}), \bb{\varepsilon}(\bb{v})), \qquad
a_\lambda(\bb{u}, \bb{v}) = ({\rm div}~\bb{u}, {\rm div}~\bb{v}).\]
In view of the Korn's inequality \eqref{Korn2}, one easily derives the the boundedness and coercivity of $a(\cdot, \cdot)$:
\begin{align*}
a(\bb{v}, \bb{w}) & \le  C |\bb{v}|_1|\bb{w}|_1, \quad \bb{v},\bb{w}\in \bb{V},  \\
a(\bb{v},\,\bb{v}) & \ge C |\bb{v}|_1^2 ,\quad \bb{v}\in \bb{V},
\end{align*}
which imply that the problem \eqref{model} has a unique solution by the Lax-Milgram lemma.

It is also well known that the regularity estimate
\begin{equation}\label{regularity}
\|\bb{u}\|_2 + \lambda\|{\rm div}\bb{u}\|_1 \lesssim \|\bb{f}\|_0
\end{equation}
holds for any convex polygonal domain $\Omega$, see \cite{Brenner-Sung-1992} for example.

\section{The locking-free virtual element method for the pure traction problem}
\label{Sec3}

In this section, we propose a locking-free virtual element method by using the reduced integration technique. This is the generalization of the locking-free finite element method introduced in \cite{Falk-1991} to the context of the virtual element methods.

\subsection{Mesh assumption and the virtual element space}

Let $\{\mathcal{T}_h\}$ be a family of decompositions of $\Omega$ into polygonal elements. The generic element is denoted by $K$ with diameter $h_K={\rm diam}(K)$.  For clarity of presentation, we first work on meshes satisfying a strong mesh assumption than the one given in \cite{Chen-Huang-2018,Brezzi-Buffa-Lipnikov-2009}:
\begin{enumerate}
	\item [{\bf C0}.]
	For each element $K$, there exist positive constant $\gamma_1, \gamma_2$ independent of $h_K$ such that  (cf. \cite{Beirao-Brezzi-Cangiani-Manzini-2013})
	\begin{itemize}
		\item
		$K$ is star-shaped with respect to a disc in $K$ with radius $\geq \gamma_1 {h_K}$;
		
		\item the distance between any two vertices of $K$ is $\ge \gamma_2 h_K$.
	\end{itemize}
\end{enumerate}

\begin{figure}[!htb]
  \centering
  \includegraphics[scale=0.5]{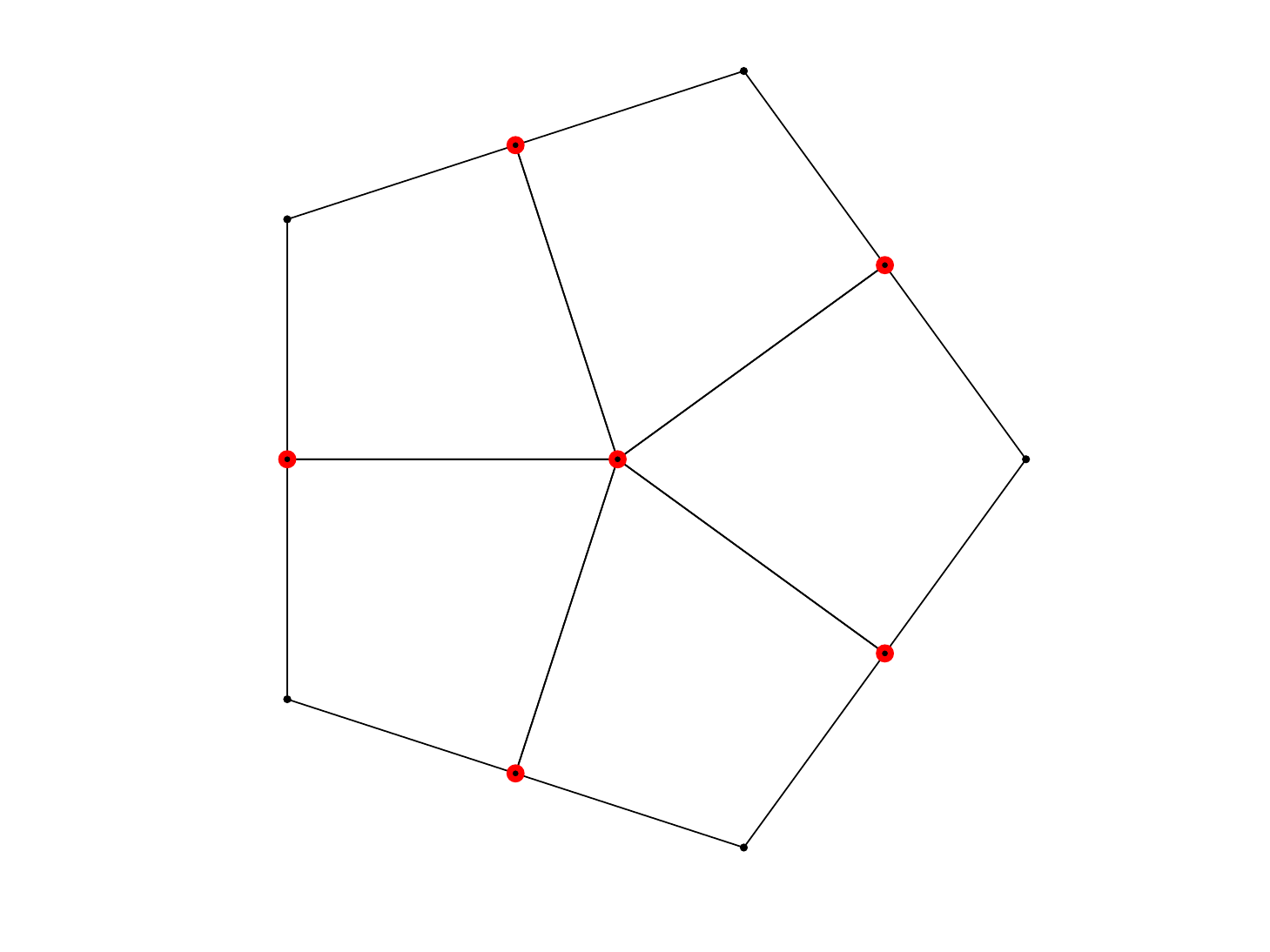}\\
  \caption{The refined element satisfying the mesh assumption {\bf C0}. The small quadrilateral element of the fine mesh and the large polygonal element of the original mesh are denoted by $E$ and $K$, respectively.}\label{Fig1_Polygon_Refine}
\end{figure}

Under the above assumption, one easily finds that the mesh can be refined by connecting the barycenter of the disc with the midpoints of the edges of $K$ as depicted in Fig.~\ref{Fig1_Polygon_Refine}. The resulting finer mesh will be denoted by $\mathcal{T}_h^{*}$ with generic element given by $E$.

In this paper, we consider the lowest order nonconforming virtual element space defined on the finer mesh $\mathcal{T}_h^*$. To this end, let's briefly review the construction proposed in \cite{Beirao-Brezzi-Cangiani-Manzini-2013}. The local virtual element space on $E$ is defined as
\begin{equation}\label{V1E}
\bb{V}_1(E)
=\big\{
\bb{v} \in \bb{H}^1(E):  \Delta \bb{v} = \bb{0} ~~\mbox{in}~~E, \quad
\partial_{\bb{n}} \bb{v} |_e \in (\mathbb{P}_0(e))^2,~~ e\subset \partial E
\big\}.
\end{equation}
The corresponding degrees of freedoms (d.o.f.s) are:
\begin{itemize}
	\item $\chi_i\in \bb{\chi}$:  the moments of $\bb{v}$ on every edge of $E$,
	\[\chi_i(\bb{v}) = \frac{1}{|e|} \int_e \bb{v} {\rm d}s, \quad e \subset \partial E.\]
\end{itemize}
Let $e\subset \partial E$ be the common edge for elements $E = E^-$ and $E^+$, and let $v$ be a scalar function defined on $e$. We introduce the jump of $v$ on $e$ by $[v] = v^- - v^+$, where $v^-$ and $v^+$ are the traces of $v$ on $e$ from the interior and exterior of $E$, respectively. For the boundary edge $e$, set $[v]|_e = v|_e$. We define the global virtual element space $\bb{V}_h$ for the pure traction problem by
\begin{align}
\bb{V}_h = \Big\{\bb{v}_h\in \bb{L}^2(\Omega): \bb{v}_h|_E \in \bb{V}_1(E), ~E \in \mathcal{T}_h^*, \quad
\int_e [\bb{v}_h] {\rm d}s = \bb{0},~ e \in \mathcal{E}_h^{*,0}, \nonumber \\
 \int_{\partial \Omega} \bb{v}_h {\rm d}s = \bb{0}, \quad
 \int_\Omega {\rm rot}_h~\bb{v}_h {\rm d}x = 0 \Big\}, \label{Vh}
\end{align}
where ${\rm rot}_h$ is the piecewise version of ${\rm rot}$ and $\mathcal{E}_h^{*,0}$ denotes the set of interior edges of $\mathcal{T}_h^*$.

Let $\bb{v}_I$ be the interpolation of $\bb{v}\in \bb{V}_h$. One can check that
 \[\int_{\partial \Omega} \bb{v}_I {\rm d}s = \bb{0}, \quad  \int_\Omega {\rm rot}_h~\bb{v}_I {\rm d}x = 0\]
 by using the integration by parts, which shows $\bb{v}_I \in \bb{V}_h$.
For later use, we also introduce the conforming virtual element space
\begin{equation}\label{V1cE}
\bb{V}_1^c(E) = \left\{ \bb{v} \in \bb{H}^1(E): \Delta \bb{v} = \bb{0} ~~{\rm in}~~E,\quad  \bb{v}|_e \in (\mathbb{P}_1(e))^2,~~ e\subset \partial E \right\}.
\end{equation}
The global conforming virtual element space is defined as
\begin{equation}\label{Vh0}
\bb{V}_h^0 = \{\bb{v}\in \bb{H}_0^1(\Omega): \bb{v} \in \bb{V}_1^c(E), ~E \in \mathcal{T}_h^* \},
\end{equation}
with d.o.f.s given by the values at the vertices of $\mathcal{T}_h^*$.

\subsection{The virtual element method based on the reduced integration technique}

For the standard lowest-order nonconforming VEMs on the coarse mesh $\mathcal{T}_h$, the discrete Korn's inequality may be invalid. We remark that when $\mathcal{T}_h$ is a triangulation of $\Omega$, the nonconforming virtual element is exactly the Crouzeix-Raviart element. In this case, the discrete Korn's inequality does not hold as proved in \cite{Falk-1991} by using a dimension-counting argument. For this reason, Falk replaces the piecewise strain tensor $\bb{\varepsilon}_h$ by a new operator $\bb{\varepsilon}_h^*$, and derive a modified finite element method with consistent convergence. The technique is referred to as the reduced integration technique in \cite{Brenner-Sung-1992}.

In this subsection, we intend to apply the reduced integration technique to the nonconforming VEMs, so as to derive a parameter-free numerical method.
Let $\bb{\varepsilon}_h$ be the piecewise strain tensor defined on the fine mesh $\mathcal{T}_h^*$. One can check that
\[\bb{\varepsilon}_h(\bb{u}) = \nabla_h \bb{u} - \frac{1}{2} {\rm rot}_h ~\bb{u}~\bb{\mathcal{X}}, \qquad
\bb{\mathcal{X}} = \begin{bmatrix} 0 & -1 \\ 1 & 0 \end{bmatrix},\]
where $\nabla_h$ and ${\rm rot}_h$ are the piecewise version of $\nabla$ and ${\rm rot}$, respectively.
The main idea is to replace the operator ${\rm rot}_h$ by the $L^2$ projection $\Pi_0 {\rm rot}_h$, which is defined on the coarse mesh $\mathcal{T}_h$, not on the fine mesh $\mathcal{T}_h^*$. Let $\Pi_0^K {\rm rot} = \Pi_0 {\rm rot}_h|_K$, where $K\in\mathcal{T}_h$. The desired operator is then piecewise defined as follows:
\begin{equation}
\begin{cases}
  \Pi_0^K{\rm rot}: \bb{V}_1(K) \to \mathbb{P}_0(K), \quad \bb v \mapsto \Pi_0^K{\rm rot}~\bb v, \\
  \int_K \Pi_0^K {\rm rot}~\bb v p {\rm d}x = \int_K {\rm rot}~\bb v p {\rm d}x,~~p \in \mathbb{P}_0(K), \label{L2ProRot}
\end{cases}
\end{equation}
where $\bb{V}_1(K) = \bb{V}_h|_K$. By the continuity of the d.o.f.s, one has
\[\int_K {\rm rot}~\bb v {\rm d}x =  \int_{\partial K} \bb v \cdot \bb{t}_K {\rm d}s, \quad \bb{v}\in \bb{V}_1(K),\]
where $\bb{t}_K = (-n_2, n_1)^\intercal$ is the anti-clockwise tangential along $\partial K$.
It is evident that the above integration can be computed by using the given d.o.f.s on the fine mesh. We finally introduce the adaptation of $\bb{\varepsilon}_h$ as
\begin{equation}\label{epsstar}
\bb{\varepsilon}_h^*(\bb{v}) = \nabla_h \bb{v} - \frac{1}{2} \Pi_0 {\rm rot}_h ~\bb{v}~\bb{\mathcal{X}}, \qquad
\bb{\mathcal{X}} = \begin{bmatrix} 0 & -1 \\ 1 & 0 \end{bmatrix}, \qquad \bb{v} \in \bb{V}_h.
\end{equation}
A direct calculation gives
\begin{align*}
(\nabla_h \bb{v} , \Pi_0 {\rm rot}_h ~\bb{w})_K
& = \int_K (-\partial_2v_1  + \partial_1v_2  ) \Pi_0 {\rm rot}_h ~\bb{w} {\rm d}x  \\
& = \int_K {\rm rot}_h ~\bb{v} \Pi_0 {\rm rot}_h ~\bb{w} {\rm d}x = (\Pi_0{\rm rot}_h ~\bb{v}, \Pi_0 {\rm rot}_h ~\bb{w})_K,
\end{align*}
and similarly,
\[(\nabla_h \bb{w} , \Pi_0 {\rm rot}_h ~\bb{v})_K = (\Pi_0{\rm rot}_h ~\bb{w}, \Pi_0 {\rm rot}_h ~\bb{v})_K.\]
We thus obtain
\begin{align}
(\bb{\varepsilon}_h^*(\bb{v}), \bb{\varepsilon}_h^*(\bb{w}))_K
& = (\nabla_h \bb{v}, \nabla_h \bb{w})_K + \frac{1}{4} (\Pi_0 {\rm rot}_h ~\bb{v} , \Pi_0 {\rm rot}_h ~\bb{w})_K \nonumber\\
&  \quad - \frac{1}{2} (\nabla_h \bb{v} , \Pi_0 {\rm rot}_h ~\bb{w})_K - \frac{1}{2} (\nabla_h \bb{w} , \Pi_0 {\rm rot}_h ~\bb{v})_K \nonumber\\
& = (\nabla_h \bb{v}, \nabla_h \bb{w})_K - \frac{1}{2} (\Pi_0 {\rm rot}_h ~\bb{v}, \Pi_0 {\rm rot}_h ~\bb{w})_K. \label{rotstar}
\end{align}

Let ${\Pi}_1^E \colon \bb{H}^1(E)\to (\mathbb{P}_1(E))^2$ be the elliptic projection of $\nabla$, which satisfies
\begin{equation}
	\label{EllipticProj}
	\begin{cases}
		(\nabla {\Pi}_1^E\bb{v} , \nabla \bb{p})_E = (\nabla \bb{v}, \nabla \bb{p})_E, \quad \bb{p}\in (\mathbb{P}_1(E))^2, \\
    \int_{\partial{E}}{\Pi}_1^E\bb{v}\,\mathrm{d}s
		= \int_{\partial{E}}\bb{v}\,\mathrm{d}s.
	\end{cases}
\end{equation}
Then we are able to introduce a discrete bilinear form with respect to $a_\mu(\bb{v}, \bb{w})$ as:
\[a_{\mu,h}(\bb{v}, \bb{w}) = \sum\limits_{K\in \mathcal{T}_h} a_{\mu,h}^K(\bb{v}, \bb{w}), \]
where
\[a_{\mu,h}^K(\bb{v}, \bb{w})
= \sum\limits_{E\subset K} \Big( (\nabla {\Pi}_1^E\bb{v} , \nabla {\Pi}_1^E\bb{w} )_E
+ S^E(\bb{v}-\nabla  \Pi_1^E\bb{v} , \bb{w}-\nabla \Pi_1^E\bb{w}) \Big)
- \frac{1}{2} ( \Pi_0^K {\rm rot}~\bb{v}, \Pi_0^K {\rm rot}~\bb{w} )_K\]
and
\[S^E(\bb{v}, \bb{w}) = \bb{\chi}(\bb{v})\cdot \bb{\chi}(\bb{w})\]
is the stabilization term frequently used in the literature of VEMs.

As usual, the second term of \eqref{bilinear} is discretized by
\[a_{\lambda,h}(\bb{v},\bb{w}) = \sum\limits_{E\in \mathcal{T}_h^*} a_{\lambda,h}^E(\bb{v}, \bb{w}), \qquad a_{\lambda,h}^E(\bb{v}, \bb{w}) = (\Pi_0^E{\rm div}~\bb{v}, \Pi_0^E{\rm div}~\bb{w})_E,\]
where the $L^2$ projection $\Pi_0^E {\rm div}$ is defined as
\begin{equation}\label{L2ProDiv}
\begin{cases}
  \Pi_0^E{\rm div}: \bb{V}_1(E) \to \mathbb{P}_0(E), \quad \bb v \mapsto \Pi_0^E{\rm div}~\bb v, \\
  \int_E \Pi_0^E{\rm div}~\bb vp {\rm d}x = \int_E {\rm div}~\bb vp {\rm d}x,~~p \in \mathbb{P}_0(E).
\end{cases}
\end{equation}

The locking-free nonconforming VEM of the variational problem \eqref{Elasticity_TensorEq_Var} is: find $\bb{u}_h\in\bb{V}_h$ such that
\begin{equation}\label{VEM}
a_h(\bb{u}_h,\,\bb{v}_h) = \langle \bb{f}_h,\,\bb{v}_h \rangle, \quad  \bb{v}_h\in\bb{V}_h,
\end{equation}
where
\begin{equation}\label{disbilinear}
a_h(\bb{u}_h,\,\bb{v}_h) = 2\mu a_{\mu,h}(\bb{u}_h, \bb{v}_h) + \lambda a_{\lambda,h}(\bb{u}_h, \bb{v}_h),
\end{equation}
and the approximation of the right hand side is given by (cf. \cite{Zhang-Zhao-Yang-Chen-2019})
\begin{equation}\label{rhs}
\langle \bb{f}_h, \bb{v}_h \rangle = (\bb{f}, P_h \bb{v}_h),
\end{equation}
where
\[P_h \bb{v}_h|_E = \frac{1}{|\partial E|} \int_{\partial E} \bb{v}_h {\rm d}s.\]

\subsection{Coercivity of the discrete bilinear form} \label{subsec:coercivity}

We first present an equivalence formula described as follows.
\begin{theorem}\label{thm:normeq}
For all $\bb{v}\in \bb{V}_h$, there holds
\[
a_{\mu,h}^K(\bb{v}, \bb{v}) \eqsim \|\bb{\varepsilon}_h^*(\bb{v})\|_{0,K}^2.
\]
\end{theorem}
\begin{proof}
According to the norm equivalence for the nonconforming virtual element functions, we have (cf. \cite{Chen-HuangX-2020,Huang-2020})
\[\|\bb{\chi}(\bb{v}-\Pi_1^E\bb{v})\|_{l^2} \eqsim |\bb{v}-\Pi_1^\nabla\bb{v}|_{1,E}.\]
By definition of the elliptic projection,
\begin{align*}
a_{\mu,h}^K(\bb{v}, \bb{v})
& \eqsim \sum\limits_{E\subset K} \Big( (\nabla {\Pi}_1^E\bb{v} , \nabla {\Pi}_1^E\bb{v} )_E
  + (\nabla(\bb{v} - {\Pi}_1^E\bb{v}) , \nabla (\bb{v} -{\Pi}_1^E\bb{v}) \Big)
- \frac{1}{2} ( \Pi_0^K {\rm rot}~\bb{v}, \Pi_0^K {\rm rot}~\bb{v} )_K \nonumber\\
& = \sum\limits_{E\subset K} (\nabla\bb{v}, \nabla \bb{v})_E
- \frac{1}{2} ( \Pi_0^K {\rm rot}~\bb{v}, \Pi_0^K {\rm rot}~\bb{v} )_K \nonumber\\
& = (\nabla_h \bb{v}, \nabla_h \bb{v})_K
- \frac{1}{2} ( \Pi_0^K {\rm rot}~\bb{v}, \Pi_0^K {\rm rot}~\bb{v} )_K.
\end{align*}
The desired formula follows from \eqref{rotstar}.
\end{proof}

For the coercivity of the proposed VEM, it remains to establish a discrete version of the Korn's inequality \eqref{Korn2} described in Theorem \ref{thm:disKorn}. The key step is to prove the following result.

\begin{lemma} \label{lem:vpexist}
Let $p\in \mathbb{P}_0(\mathcal{T}_h)$ be piecewise constant on the coarse mesh $\mathcal{T}_h$, satisfying $\int_\Omega p {\rm d}x = 0$. Then there exists a conforming virtual element function $\bb{v} \in \bb{V}_h^0$ defined on the fine mesh $\mathcal{T}_h^*$ such that
\[\int_\Omega {\rm div} \bb{v} q {\rm d} x = \int_\Omega p q {\rm d} x, \quad q \in \mathbb{P}_0(\mathcal{T}_h)\]
and
\[\|\bb{v}\|_1 \le C \|p\|_0,\]
where the constant $C$ is independent of $\bb{v}$ and $p$.
\end{lemma}
\begin{proof}
According to the well-known Lemma 11.2.3 in \cite{BrennerScott2008}, there exists $\bb{u} \in \bb{H}_0^1(\Omega)$ such that
\[
{\rm div}~\bb{u} = p, \qquad \|\bb{u}\|_1 \le C \|p\|_0.
\]
Then we only need to prove that there exists $\bb{v} \in \bb{V}_h^0$ such that
\[\int_\Omega {\rm div}~\bb{v}q {\rm d} x  = \int_\Omega {\rm div}~ \bb{u}q {\rm d} x, \quad q \in \mathbb{P}_0(\mathcal{T}_h),\]
or equivalently,
\begin{equation}\label{commute}
\int_K {\rm div} ~\bb{v} {\rm d} x  = \int_K {\rm div} ~\bb{u}  {\rm d} x, \quad K \in \mathcal{T}_h.
\end{equation}

We construct the function $\bb{v}$ by determining its d.o.f. values in $\bb{V}_h^0$. For the vertex and the center of the disc (cf. Fig.~\ref{Fig1_Polygon_Refine}), denoted by $a$, we define
\[\bb{v}(a) = \bb{u}(a).\]
For the middle point $z$ of an edge $e$ in the coarse mesh, let $z_1$ and $z_2$ be the endpoints of $e$. We define
\[\bb{v}(z) = \frac{2}{|e|} \int_e \bb{u} {\rm d}s - \frac{\bb{u}(z_1) + \bb{u}(z_2)}{2}.\]
One easily finds that $\bb{v} \in \bb{V}_h^0$ satisfies
\[\int_e \bb{v} {\rm d}s = \int_e \bb{u} {\rm d}s, \quad e \subset \partial K, ~~ K \in \mathcal{T}_h,\]
which implies \eqref{commute} by using the integration by parts.

Let $I_h^0\bb{u}$ be the interpolation of $\bb{u}$ in $\bb{V}_h^0$. The triangle inequality gives
\[|\bb{v}-\bb{u}|_{1,K} \le |\bb{v}-I_h^0\bb{u}|_{1,K} + |\bb{u}-I_h^0\bb{u}|_{1,K}, \quad K \in \mathcal{T}_h.\]
By the standard interpolation error estimates,
\[h_K^{-1} \|\bb{u}-I_h^0\bb{u}\|_{0,K} + |\bb{u}-I_h^0\bb{u}|_{1,K} \lesssim |\bb{u}|_{1,K}.\]
Let $\bb{\chi}^c$ be the d.o.f vector on $K$. According to the inverse inequality and norm equivalence in \cite{Chen-Huang-2018}, and the definition of $\bb{v}$, we obtain
\[
|\bb{v}-I_h^0\bb{u}|_{1,K}
 \lesssim h_K^{-1} \|\bb{v}-I_h^0\bb{u}\|_{0,K} \eqsim \|\bb{\chi}^c (\bb{v}-I_h^0\bb{u})\|_{l^2} \\
 = \Big(\sum\limits_{i=1}^n (\bb{v}(z_i) - \bb{u}(z_i))^2\Big)^{1/2},
\]
where $z_i$ are the midpoints of the edges of $K$ and $n$ is the number of the edges. For fixed $z = z_i$, let $e$ be the corresponding edge with endpoints $z_1$ and $z_2$. Since $I_h^0\bb{u}$ is piecewise linear along $\partial K$, we have
\begin{align*}
\bb{v}(z) - \bb{u}(z)
& =  \frac{2}{|e|} \int_e \bb{u} {\rm d}s - \frac{\bb{u}(z_1) + \bb{u}(z_2)}{2} - \bb{u}(z) \\
& = \frac{2}{|e|} \int_e \bb{u} {\rm d}s - \frac{2}{|e|} \int_e I_h^0 \bb{u} {\rm d}s = \frac{2}{|e|} \int_e (\bb{u}-I_h^0\bb{u}) {\rm d}s,
\end{align*}
which along with the Cauchy-Schwarz inequality and the trace inequality (cf. \cite{BrennerScott2008,Chen-Huang-2018}) yields
\[
|\bb{v}(z) - \bb{u}(z)|
 \lesssim h_e^{-1/2} \|\bb{u}-I_h^0\bb{u}\|_{0,e} \lesssim h_K^{-1} \|\bb{u}-I_h^0\bb{u}\|_{0,K} + |\bb{u}-I_h^0\bb{u}|_{1,K}.
\]
The desired estimate follows by combining the previous inequalities.
\end{proof}

With the help of Lemma \ref{lem:vpexist}, We are in a position to prove the discrete Korn's inequality following a similar argument in \cite[Theorem 6.1]{Falk-1991}. Considering that this problem is in the context of VEMs, we still give the proof for the sake of completeness.

\begin{theorem} \label{thm:disKorn}
For all $\bb{v} \in \bb{V}_h$, there exists a constant $C$ independent of $\bb{v}$ such that
\[|\bb{v}|_{1,h} \le C \|\bb{\varepsilon}_h^*(\bb{v})\|_{0,h}.\]
\end{theorem}
\begin{proof}
By definition, we have for all $\bb{\tau}\in \bb{L}^2(\Omega)$,
\[\int_\Omega \bb{\varepsilon}_h^*(\bb{v}) : \bb{\tau} {\rm d} x
= \int_\Omega \Big(\nabla_h \bb{v} - \frac{1}{2} \Pi_0 {\rm rot}_h \bb{v}~\bb{\mathcal{X}}\Big) : \bb{\tau} {\rm d} x.\]
Using Lemma \ref{lem:vpexist}, we may choose $\bb{\tau} = \nabla_h \bb{v} - {\rm curl}~ \bb{z}$, where
\[{\rm curl}~ \bb{z} = \begin{bmatrix}
\partial_2 z_1  &  - \partial_1 z_1 \\  \partial_2 z_2  &  -\partial_1 z_2
\end{bmatrix}, \quad \bb{z} = (z_1,z_2)^\intercal\]
and $\bb{z}\in \bb{V}_h^0$ satisfies
\begin{equation}\label{rotzexs}
\int_\Omega {\rm div} \bb{z} q {\rm d} x = \int_\Omega {\rm rot}_h \bb{v} q {\rm d} x, \quad q \in \mathbb{P}_0(\mathcal{T}_h);
\qquad \|\bb{z}\|_1 \le C \|{\rm rot}_h \bb{v}\|_0.
\end{equation}
Then
\begin{equation}\label{tauL2}
\|\bb{\tau}\|_0  \le \|\nabla_h \bb{v}\|_0 + \|{\rm curl}~ \bb{z}\|_0 \le C(\|\nabla_h \bb{v}\|_0 + \|{\rm rot}_h \bb{v}\|_0)
\le C \|\nabla_h \bb{v}\|_0 = C|\bb{v}|_{1,h}.
\end{equation}
Using the integration by parts and noting that ${\rm div} ({\rm curl}~\bb{z}) = \bb{0}$,  one has
\[\int_\Omega \nabla_h \bb{v} : {\rm curl}~ \bb{z} {\rm d} x
= \sum\limits_{E\in \mathcal{T}_h^*} \int_{\partial E} \bb{v} \cdot \frac{\partial \bb{z}}{\partial s}{\rm d}s
+ \sum\limits_{E\in \mathcal{T}_h^*} \int_E  \bb{v} \cdot {\rm div} ({\rm curl}~\bb{z}) {\rm d}x
= \sum\limits_{E\in \mathcal{T}_h^*} \int_{\partial E} \bb{v} \cdot \frac{\partial \bb{z}}{\partial s}{\rm d}s  = 0,\]
where the summation vanishes since on the boundary edges $\bb{z} = 0$, while on interior edges, contributions from adjoining elements cancel by the continuity of the d.o.f.s. Here we have used the fact that the tangential derivatives of $\bb{z}$ are polynomials of degree 0 by the definition of $\bb{V}_h^0$. According to the $L^2$ orthogonality of $\nabla_h \bb{v}$  and ${\rm curl}~\bb{z}$ and \eqref{rotzexs}, we obtain
\begin{align*}
\int_\Omega \bb{\varepsilon}_h^*(\bb{v}) : \bb{\tau} {\rm d} x
& = \int_\Omega \Big(\nabla_h \bb{v} - \frac{1}{2} \Pi_0 {\rm rot}_h \bb{v}~\bb{\mathcal{X}}\Big) : (\nabla_h \bb{v} - {\rm curl}~ \bb{z}) {\rm d} x \\
& = \int_\Omega \nabla_h \bb{v} : \nabla_h \bb{v}{\rm d} x
  - \frac{1}{2}  \int_\Omega  \Pi_0 {\rm rot}_h \bb{v}~\bb{\mathcal{X}}  : (\nabla_h \bb{v} - {\rm curl}~ \bb{z})  {\rm d} x \\
& = \int_\Omega \nabla_h \bb{v} : \nabla_h \bb{v}{\rm d} x
  - \frac{1}{2}  \int_\Omega  \Pi_0 {\rm rot}_h \bb{v} ({\rm rot}_h \bb{v} - {\rm div}~ \bb{z})  {\rm d} x \\
& = \int_\Omega \nabla_h \bb{v} : \nabla_h \bb{v}{\rm d} x
  - \frac{1}{2}  \int_\Omega  \Pi_0 {\rm rot}_h \bb{v} ({\rm rot}_h \bb{v} - {\rm rot}_h \bb{v})  {\rm d} x
  = \|\nabla_h \bb{v}\|_0^2 = |\bb{v}|_{1,h}^2,
\end{align*}
which together with the estimate of \eqref{tauL2} yields
\[
\|\bb{\varepsilon}_h^*(\bb{v})\|_0
\ge \frac{\int_\Omega \bb{\varepsilon}_h^*(\bb{v}) : \bb{\tau} {\rm d} x}{\|\bb{\tau}\|_0}
\ge C |\bb{v}|_{1,h},
\]
as required.
\end{proof}

We now have the following coercivity result for the discrete bilinear form.
\begin{theorem}\label{thm:coercivity}
For all $\bb{v}\in \bb{V}_h$, there holds
\[|\bb{v}|_{1,h}^2 \lesssim a_h(\bb{v}, \bb{v}).\]
\end{theorem}
\begin{proof}
This is a direct consequence of Theorems \ref{thm:normeq} and \ref{thm:disKorn} according to the definition of the discrete bilinear form, see Eq. \eqref{disbilinear}.
\end{proof}

\subsection{Error analysis of the VEM}

Following the similar arguments in \cite{Beirao-Brezzi-Marini-2013,Zhang-Zhao-Yang-Chen-2019}, we can derive an abstract lemma for error analysis described as follows.
\begin{lemma}	\label{Thm1_Elasticity_TensorEq_VEM_Conver_Est}
Let $\bb{u}_h\in\bb{V}_h$ be the solution of the discrete problem \eqref{VEM} and $\bb{u}$ the weak solution of problem \eqref{model}. Then under the mesh assumption \textbf{C0}, for any piecewise polynomial $\bb{u}_{\pi}\in (\mathbb{P}_k(\mathcal{T}_h^*))^2$, there holds
	\begin{equation}\label{Strang}
	| \bb{u} - \bb{u}_h |_{1,h}
	\lesssim  | \bb{u} - \bb{u}_I |_{1,h}
	+ | \bb{u} - \bb{u}_{\pi} |_{1,h}
	+ \lambda \| {\rm div}\bb{u} - \Pi_0{\rm div} \bb{u} \|_0
	+ \| \bb{f} - \bb{f}_h \|_{\bb{V}_h^{'}} + E_h,
	\end{equation}
	where $\bb{u}_I$ is the interpolation of $\bb{u}$ in $\bb{V}_h$, and
\[\| \bb{f} - \bb{f}_h \|_{\bb{V}_h^{'}}
	= \sup_{\bb{v}\in\bb{V}_h}
	\frac{|(\bb{f},\,\bb{v}_h) - \langle\bb{f}_h,\,\bb{v}_h\rangle|}{|\bb{v}_h|_{1,h}},
\quad E_h = \mathop {\sup}\limits_{\boldsymbol{v}_h \in \boldsymbol{V}_h} \frac{|a(\boldsymbol{u},\boldsymbol{v}_h) - (\boldsymbol{f},\boldsymbol{v}_h)|}{|\bb{v}_h|_{1,h}}.
\]
\end{lemma}
\begin{proof}
By setting $\bb{\delta}_h = \bb{u}_h-\bb{u}_I$ and using the triangle's inequality, it suffices to bound $|\bb{\delta}_h|_{1,h}$. According to Theorem \ref{thm:coercivity}, we have the coercivity
\[|\bb{\delta}_h|_{1,h}^2 \lesssim a_h(\bb{\delta}_h, \bb{\delta}_h).\]
The rest of the argument is standard, so we omit it for simplicity. One can refer to Theorem 6.1 in \cite{Zhang-Zhao-Yang-Chen-2019} for details.
\end{proof}

With the help of the abstract lemma, we are able to derive the following estimate.

\begin{theorem}\label{thm:err}
For all $\bb{f}\in \bb{L}^2(\Omega)$, let $\bb{u}_h\in\bb{V}_h$ be the solution of the discrete problem \eqref{VEM} and $\bb{u}$ the weak solution of problem \eqref{model}. There holds
\[|\bb{u}-\bb{u}_h|_{1,h} \lesssim h \|\bb{f}\|_0.\]
\end{theorem}
\begin{proof}
According to the interpolation and the projection error estimates, one easily obtains
\[| \bb{u} - \bb{u}_I |_{1,h} + | \bb{u} - \bb{u}_{\pi} |_{1,h}  \lesssim h \|\bb{f}\|_0,\]
where $\bb{u}_{\pi}$ can be chosen as the piecewise $L^2$ or elliptic projection. By the regularity estimate \eqref{regularity},
\[\lambda \| {\rm div}\bb{u} - \Pi_0{\rm div} \bb{u} \|_0 \lesssim  h \lambda\|{\rm div} \bb{u}\|_0 \lesssim h \|\bb{f}\|_0.\]
The estimates of the last two terms in \eqref{Strang} can be found in \cite{Zhang-Zhao-Yang-Chen-2019}. This completes the proof.
\end{proof}

\begin{remark}
  Using the standard duality argument, we may conclude that
 \[\|\bb{u}-\bb{u}_h\|_0 \lesssim h^2 \|\bb{f}\|_0\]
 for the convex polygonal domain $\Omega$.
\end{remark}

\section{Some discussions} \label{sec:discussion}

\subsection{More general mesh assumption}

It is obvious that the lowest-order nonconforming virtual element on triangular meshes is exactly the Crouzeix-Raviart finite element. In this case, if the mesh refinement in Fig.~\ref{Fig1_Polygon_Refine} or Fig.~\ref{fig:refineType}~(a) is replaced by the second one in Fig.~\ref{fig:refineType}, then our method is reduced to the one given by Falk in \cite{Falk-1991}.

\begin{figure}[!htb]
  \centering
  \subfigure[refineType=1]{\includegraphics[scale=0.35]{images/refineType1-eps-converted-to.pdf}}
  \subfigure[refineType=2]{\includegraphics[scale=0.35]{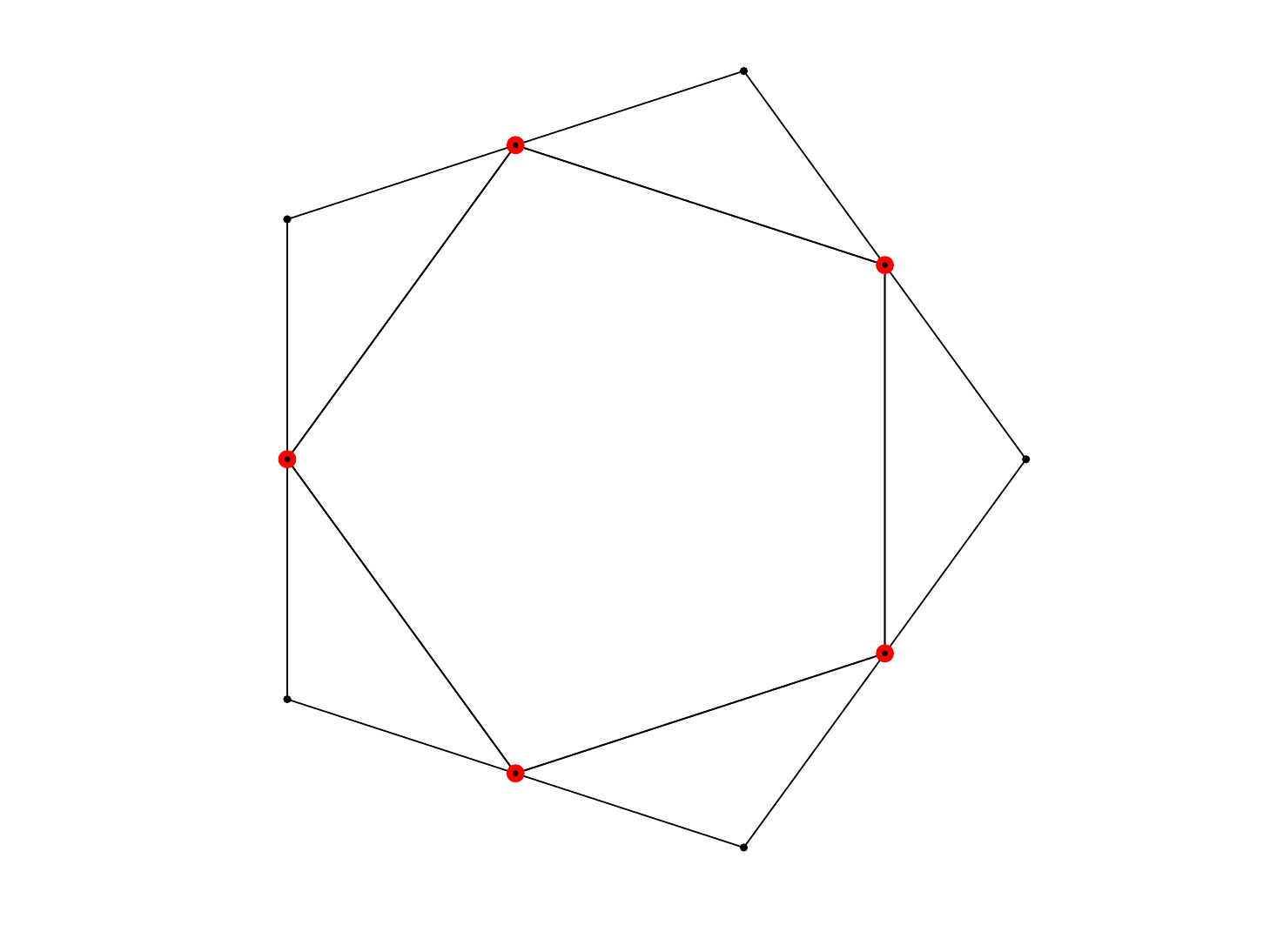}}
  \subfigure[refineType=3]{\includegraphics[scale=0.35]{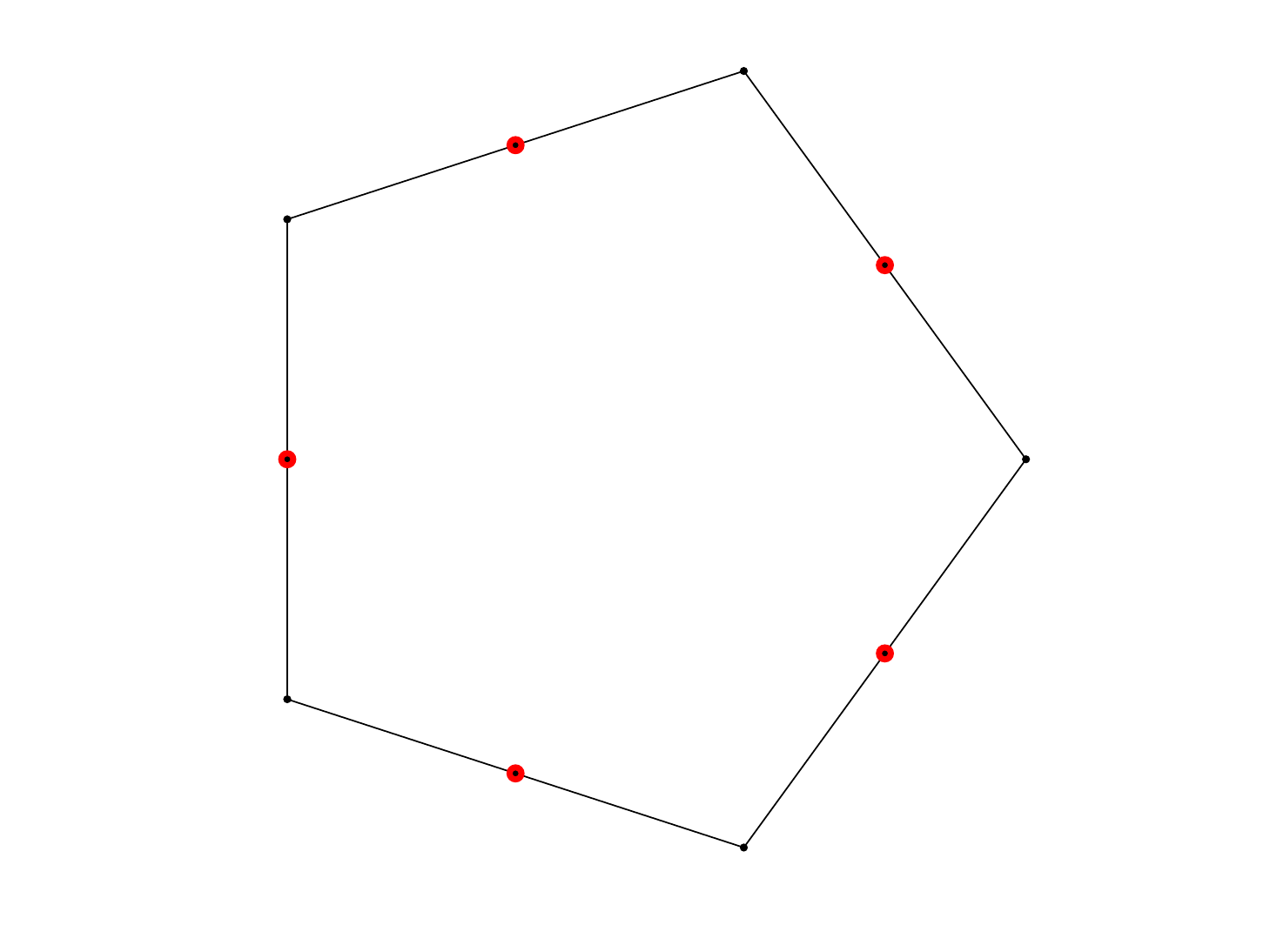}}\\
  \caption{Three types of mesh refinement}\label{fig:refineType}
\end{figure}

Our analysis also applies for the mesh assumption given in \cite{Chen-Huang-2018,Brezzi-Buffa-Lipnikov-2009} (even more general mesh assumptions), where each polygon admits a virtual quasi-uniform and regular triangulation. In this case, we refine the mesh as in Fig.~\ref{fig:refineType}~(c) by just adding midpoints on each edges, giving rise to a new mesh with each polygon having hanging nodes or collinear edges. It is evident that the construction in Lemma \ref{lem:vpexist} is still valid, which therefore implies the locking-free property and the optimal rates of convergence.

\subsection{The VEM in the enhancement virtual element space}

We can also consider the problem in the frequently used space $\widehat{\bb{H}}(\Omega)$ (see Eq.~\eqref{Hhat}). To this end, we first introduce a lifting space
\[
\widetilde{\bb{V}}_1(E) =\big\{
\bb{v} \in \bb{H}^1(E): \Delta \bb{v} \in (\mathbb{P}_0(E))^2, ~~ \partial_{\bb{n}}\bb{v} |_e \in (\mathbb{P}_0(e))^2, ~~e \subset \partial{E}\big\}
\]
with the degrees of freedom
\begin{itemize}
  \item the moments $\frac{1}{|e|} \int_e \bb{v} {\rm d}s$, $e \subset \partial E$;
  \item the moments $\frac{1}{|E|} \int_E \bb{v} {\rm d}x$.
\end{itemize}
and define the elliptic projection $\Pi_1^E$ as before. One easily finds that the redundant moments on each element are not involved in the computation of $\Pi_1^E$. We therefore introduce an enhancement virtual element space as (cf. \cite{Ahmad-Alsaedi-Brezzi-Marini-2013,Zhang-Zhao-Yang-Chen-2019})
\begin{equation}\label{W1E}
\bb{W}_1(E) =\Big\{
\bb{v} \in \widetilde{\bb{V}}_1(E):  \int_E \bb{v} {\rm d}x = \int_E \Pi_1^E \bb{v} {\rm d}x \Big\},
\end{equation}
with the global space given by
\begin{align}
\bb{W}_h = \Big\{\bb{v}_h\in \bb{L}^2(\Omega): \bb{v}_h|_E \in \bb{W}_1(E), ~E \in \mathcal{T}_h^*, \quad
\int_e [\bb{v}_h] {\rm d}s = \bb{0},~ e \in \mathcal{E}_h^{*,0}, \nonumber \\
 \int_{\Omega} \bb{v}_h {\rm d}s = \bb{0}, \quad
 \int_\Omega {\rm rot}_h~\bb{v}_h {\rm d}x = 0 \Big\}, \label{enhancement}
\end{align}
which obviously shares the same d.o.f.s with $\bb{V}_h$, i.e., the moments $\frac{1}{|e|} \int_e \bb{v} {\rm d}s$, $e \in \mathcal{E}_h^*$.
Our analysis still holds in such case since the first constraint is not used in the proof of the discrete Korn's inequality in Theorem \ref{thm:disKorn}.

\subsection{The locking-free VEM for the pure displacement problem}

The linear elasticity problem in the pure displacement formulation is
\[
\begin{cases}
- {\rm div}~\bb{\sigma}(\bb{u}) = \bb{f}  \quad & \mbox{in} ~~\Omega, \\
\bb{u}  = \bb{0}. \quad & \mbox{on} ~~\partial\Omega.\\
\end{cases}
\]
The locking-free VEM has the same formulation as the one given in \eqref{VEM}. One just needs to replace the global nonconforming virtual element space by
\begin{align}
\bb{U}_h = \Big\{\bb{v}_h\in \bb{L}^2(\Omega): \bb{v}_h|_E \in \bb{V}_1(E), ~E \in \mathcal{T}_h^*, \quad
\int_e [\bb{v}_h] {\rm d}s = \bb{0},~ e \in \mathcal{E}_h^* \label{Uh}
 \Big\},
\end{align}
where $\mathcal{E}_h^*$ denotes the set of all edges in $\mathcal{T}_h^*$. According to the continuity of the d.o.f.s, one easily obtains
\begin{equation}\label{rotint}
\int_\Omega {\rm rot}_h~\bb{v}_h {\rm d}x = 0
\end{equation}
for all $\bb{v}_h\in \bb{U}_h$, which is the crucial constraint in establishing the discrete Korn's inequality in Theorem \ref{thm:disKorn}.
For this reason, the arguments in Subsection \ref{subsec:coercivity} still hold, naturally leading to the robustness with respect to the Lam\'{e} constant and the optimal rates of convergence.

\subsection{A unified locking-free scheme both for the conforming and nonconforming VEMs}

In the literature of VEMs, the local bilinear form $({\rm div}~\bb{v}, {\rm div}~\bb{w})_E$ of the continuous variational problem
is usually approximated by $(\Pi_0^E{\rm div}~\bb{v}, \Pi_0^E{\rm div}~\bb{w})_E$ as in \eqref{L2ProDiv}, which, however, is proved to be locking-free only for $k\ge 2$ in \cite{Beirao-Brezzi-Marini-2013} for the conforming VEMs. By using the reduced integration technique, we in \cite{Huang-Lin-Yu-arxiv} proposed a conforming locking-free VEM in the lowest order case, where $\sum\limits_{E\subset K}(\Pi_0^E{\rm div}~\bb{v}, \Pi_0^E{\rm div}~\bb{w})_E$ is replaced by $(\Pi_0^K{\rm div}~\bb{v}, \Pi_0^K{\rm div}~\bb{w})_K$,  with $\Pi_0^K{\rm div}$ defined as
\begin{equation} \label{Pi0Kdiv}
\begin{cases}
\Pi_0^K{\rm div}: \bb{V}_1(K) \to \mathbb{P}_0(K), \quad \bb{v} \mapsto \Pi_0^K{\rm div}~\bb{v},  \\%\nonumber
\int_K (\Pi_0^K{\rm div}~\bb{v})p\,\mathrm{d}x = \int_K ({\rm div}~\bb{v})p\,\mathrm{d}x,\quad p \in \mathbb{P}_0(K),
\end{cases}
\end{equation}
which is similar to the definition of $\Pi_0^K{\rm rot}$ given in \eqref{L2ProRot}.
Inspired by the above treatments in \eqref{L2ProRot} and \eqref{Pi0Kdiv}, we intend to propose a unified locking-free scheme both for the conforming and nonconforming VEMs described as follows.

Let $\bb{S}_h$ be the global virtual element space. For clarity, we list the choices of $\bb{S}_h$ in the following:
\begin{enumerate}
  \item \textbf{Nonconforming VEMs}
  \begin{itemize}
    \item Pure traction problem in the original space (cf. \eqref{V1E} and \eqref{Vh})
    \begin{align*}
   \bb{S}_h = \bb{V}_h = \Big\{\bb{v}_h\in \bb{L}^2(\Omega): \bb{v}_h|_E \in \bb{V}_1(E), ~E \in \mathcal{T}_h^*, \quad
   \int_e [\bb{v}_h] {\rm d}s = \bb{0},~ e \in \mathcal{E}_h^{*,0},  \\
   \int_{\partial \Omega} \bb{v}_h {\rm d}s = \bb{0}, \quad
   \int_\Omega {\rm rot}_h~\bb{v}_h {\rm d}x = 0 \Big\}.
  \end{align*}
    \item Pure traction problem in the enhancement space (cf. \eqref{W1E} and \eqref{enhancement})
    \begin{align*}
  \bb{S}_h = \bb{W}_h = \Big\{\bb{v}_h\in \bb{L}^2(\Omega): \bb{v}_h|_E \in \bb{W}_1(E), ~E \in \mathcal{T}_h^*, \quad
\int_e [\bb{v}_h] {\rm d}s = \bb{0},~ e \in \mathcal{E}_h^{*,0},  \\
 \int_{\Omega} \bb{v}_h {\rm d}s = \bb{0}, \quad
 \int_\Omega {\rm rot}_h~\bb{v}_h {\rm d}x = 0 \Big\}.
\end{align*}
    \item Pure displacement problem (cf. \eqref{Uh})
    \begin{align*}
  \bb{S}_h = \bb{U}_h = \Big\{\bb{v}_h\in \bb{L}^2(\Omega): \bb{v}_h|_E \in \bb{V}_1(E), ~E \in \mathcal{T}_h^*, \quad
\int_e [\bb{v}_h] {\rm d}s = \bb{0},~ e \in \mathcal{E}_h^*
 \Big\}.
\end{align*}
  \end{itemize}
  \item \textbf{Conforming VEMs}
    \begin{itemize}
    \item Pure traction problem in the original space (cf. \eqref{V1cE} and \eqref{Vh})
    \begin{align*}
   \bb{S}_h = \bb{V}_h^c = \Big\{\bb{v}_h\in \bb{H}^1(\Omega): \bb{v}_h|_E \in \bb{V}_1^c(E), ~E \in \mathcal{T}_h^*, \quad
   \int_{\partial \Omega} \bb{v}_h {\rm d}s = \bb{0}, \quad
   \int_\Omega {\rm rot}~\bb{v}_h {\rm d}x = 0 \Big\}.
  \end{align*}
    \item Pure traction problem in the enhancement space
    \begin{align*}
    \bb{S}_h = \bb{W}_h^c = \Big\{\bb{v}_h\in \bb{H}^1(\Omega): \bb{v}_h|_E \in \bb{W}_1^c(E), ~E \in \mathcal{T}_h^*, \quad
   \int_{\Omega} \bb{v}_h {\rm d}s = \bb{0}, \quad
   \int_\Omega {\rm rot}~\bb{v}_h {\rm d}x = 0 \Big\},
  \end{align*}
 where
 \begin{equation*}
\bb{W}_1^c(E) = \Big\{
\bb{v} \in \bb{H}^1(E): \Delta \bb{v} \in (\mathbb{P}_0(E))^2,  \quad
  \int_E \bb{v} {\rm d}x = \int_E \Pi_1^E \bb{v} {\rm d}x\Big\}.
\end{equation*}
    \item Pure displacement problem (cf. \eqref{Vh0})
  \[
  \bb{S}_h = \bb{U}_h^c = \bb{V}_h^0 = \{\bb{v}\in \bb{H}_0^1(\Omega): \bb{v} \in \bb{V}_1^c(E), ~E \in \mathcal{T}_h^* \}.
  \]
  \end{itemize}
\end{enumerate}

The unified locking-free VEM of the variational problem \eqref{Elasticity_TensorEq_Var} is: find $\bb{u}_h\in\bb{S}_h$ such that
\begin{equation}\label{unifiedVEM}
a_h(\bb{u}_h, \bb{v}_h) = \langle \bb{f}_h,\,\bb{v}_h \rangle, \quad  \bb{v}_h\in\bb{S}_h,
\end{equation}
where
\[
a_h(\bb{u}_h, \bb{v}_h) = 2\mu a_{\mu,h}(\bb{u}_h, \bb{v}_h) + \lambda a_{\lambda,h}(\bb{u}_h, \bb{v}_h),
\]
with
\[a_{\mu,h}^K(\bb{v}, \bb{w})
= \sum\limits_{E\subset K} \Big( (\nabla {\Pi}_1^E\bb{v} , \nabla {\Pi}_1^E\bb{w} )_E
+ S^E(\bb{v}-\nabla  \Pi_1^E\bb{v} , \bb{w}-\nabla \Pi_1^E\bb{w}) \Big)
- \frac{1}{2} ( \Pi_0^K {\rm rot}~\bb{v}, \Pi_0^K {\rm rot}~\bb{w} )_K,\]
\[a_{\lambda,h}^K(\bb{v}, \bb{w}) = (\Pi_0^K{\rm div}~\bb{v}, \Pi_0^K{\rm div}~\bb{w})_E, \qquad S^E(\bb{v}, \bb{w}) = \bb{\chi}(\bb{v})\cdot \bb{\chi}(\bb{w}).\]
The approximation of the right hand side is the same as \eqref{rhs}.

Let us briefly analyze the locking-free property and the optimal error estimates. For the bilinear form $a_{\mu,h}^K(\bb{v}, \bb{w})$, the crucial step is to establish the discrete Korn's inequality in Theorem \ref{thm:disKorn} so as to ensure the coercivity in Lax-Milgram lemma, which has been justified for the nonconforming methods. For the conforming methods, it is obvious that the constraint \eqref{rotint} holds for all cases, thus yielding the Korn's inequality. With the help of the coercivity, we can derive the following abstract lemma following a standard calculation.
      \begin{lemma}	\label{lem:Strangall}
Let $\bb{u}_h\in\bb{V}_h$ be the solution of the discrete problem \eqref{unifiedVEM} and $\bb{u}$ the weak solution of problem \eqref{model}. Then for all $\bb{u}_{\mathcal{I}} \in \bb{S}_h$ and for any piecewise polynomial $\bb{u}_{\pi}\in (\mathbb{P}_k(\mathcal{T}_h^*))^2$, there holds
	\begin{equation}\label{Strang}
	| \bb{u} - \bb{u}_h |_{1,h}
	\lesssim  | \bb{u} - \bb{u}_{\mathcal{I}} |_{1,h}
	+ | \bb{u} - \bb{u}_{\pi} |_{1,h}
	+ \lambda \| {\rm div}\bb{u} - \Pi_0{\rm div} \bb{u}_{\mathcal{I}} \|_0
	+ \| \bb{f} - \bb{f}_h \|_{\bb{V}_h^{'}} + E_h,
	\end{equation}
	where
\[\| \bb{f} - \bb{f}_h \|_{\bb{V}_h^{'}}
	= \sup_{\bb{v}\in\bb{V}_h}
	\frac{|(\bb{f},\,\bb{v}_h) - \langle\bb{f}_h,\,\bb{v}_h\rangle|}{|\bb{v}_h|_{1,h}},
\quad E_h = \mathop {\sup}\limits_{\boldsymbol{v}_h \in \boldsymbol{V}_h} \frac{|a(\boldsymbol{u},\boldsymbol{v}_h) - (\boldsymbol{f},\boldsymbol{v}_h)|}{|\bb{v}_h|_{1,h}}.
\]
\end{lemma}

For the second term $a_{\lambda,h}^K(\bb{v}, \bb{w})$, the fundamental idea is to prove that there exists a VEM function $\bb{u}_{\mathcal{I}} \in \bb{S}_h$ such that
  \[\Pi_0^K{\rm div} \bb{u}_{\mathcal{I}} = \Pi_0^K{\rm div} \bb{u}, \quad
  |\bb{u}-\bb{u}_{\mathcal{I}}|_{1,h} \lesssim h |\bb{u}|_2,\]
  where $\bb{u}\in \bb{H}^2(\Omega)$ is the exact solution.
  The above result is trivial for the nonconforming methods since we can take $\bb{u}_{\mathcal{I}} = \bb{u}_I$ to be the interpolation of $\bb{u}$ in $\bb{S}_h$. For the conforming methods, please refer to \cite{Huang-Lin-Yu-arxiv}. Thus, the regularity estimate \eqref{regularity} yields
  \[\lambda \| {\rm div}\bb{u} - \Pi_0{\rm div} \bb{u}_{\mathcal{I}} \|_0
   = \lambda \| {\rm div}\bb{u} - \Pi_0{\rm div} \bb{u} \|_0 \lesssim \lambda h \|{\rm div} \bb{u}\|_0
   \lesssim h\|\bb{f}\|_0,
   \]
   as required.

\section{Numerical examples}

In this section, we report the performance of our proposed virtual element method by testing the accuracy and the robustness with respect to the Lam\'{e} constant $\lambda$. For simplicity, we only consider the nonconforming VEMs. One can refer to \cite{Huang-Lin-Yu-arxiv} for the test of the conforming methods.
Unless otherwise specified, the domain $\Omega$ is taken as the unit square $(0,1)^2$, and the Lam\'{e} constants are set as $\lambda = 10^{10}$ and $\mu=1$. All examples are implemented in MATLAB R2019b. Our code is available from GitHub (\url{https://github.com/Terenceyuyue/mVEM}) as part of the mVEM package which contains efficient and easy-following codes for various VEMs published in the literature. The subroutine {\color{gray} elasticityVEM\_NCreducedIntegration.m} is used to compute the numerical solutions and the test script {\color{gray} main\_elasticityVEM\_NCreducedIntegration.m} verifies the convergence rates.

\begin{figure}[!htb]
  \centering
  \subfigure{\includegraphics[scale=0.55,trim = 50 0 50 0,clip]{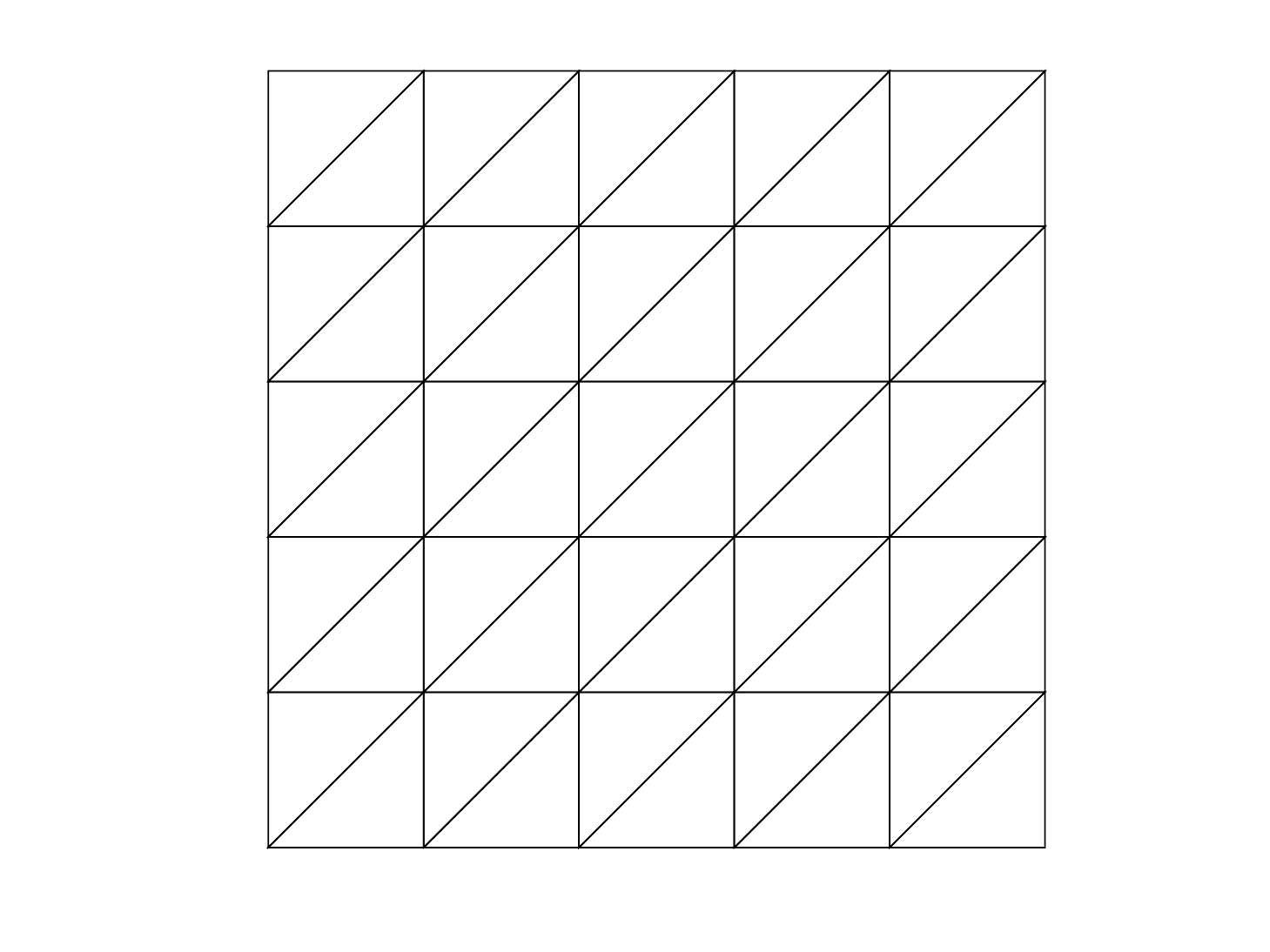}}
  \subfigure{\includegraphics[scale=0.55,trim = 50 0 50 0,clip]{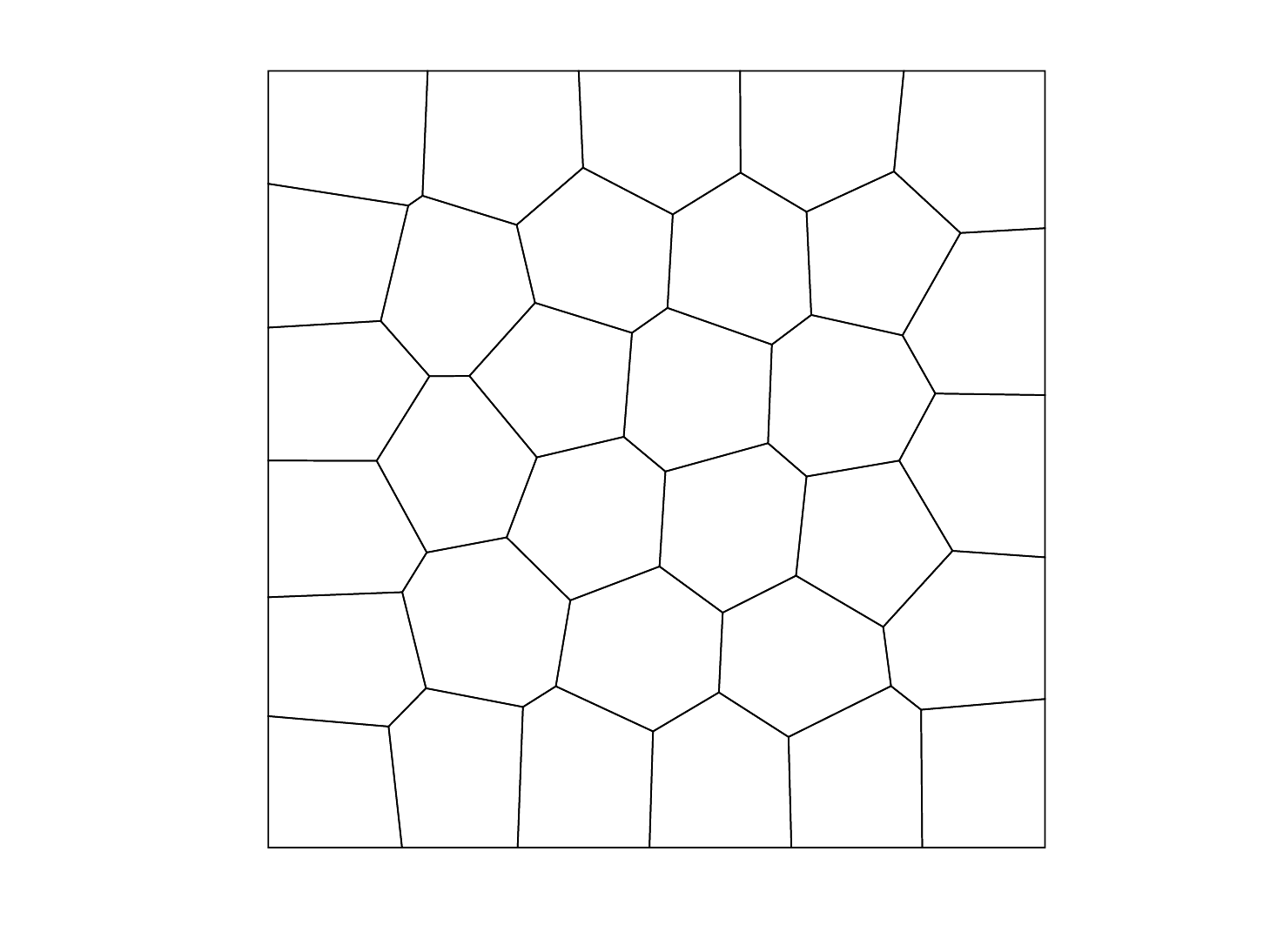}}\\
  \caption{The uniform triangular (left) and unstructured polygonal (right) meshes.}\label{fig:mesh}
\end{figure}

\subsection{Implementation of the proposed VEM} \label{subsec:implementation}

The discrete variational problem is given by \eqref{VEM}, where the constraints
\[\int_{\partial \Omega} \bb{u}_h{\rm d}s = \bb{0} \quad \mbox{and} \quad
\int_\Omega {\rm rot}_h~\bb{u}_h {\rm d}x = 0\]
are not naturally imposed. To do so, we introduce Lagrange multipliers $\beta_i \in \mathbb{R}$ and consider the augmented variational formulation: Find $(\boldsymbol{u}_h,\beta_1, \beta_2,\beta_3)\in \bb{V}_h \times \mathbb{R}\times \mathbb{R}\times \mathbb{R}$ such that
\begin{equation}\label{augVEM}
\begin{cases}
a_h(\bb{u}_h,\bb{v}_h) + \beta_1\int_{\partial \Omega} \bb{u}_{h,1}{\rm d}s
  + \beta_2\int_{\partial \Omega} \bb{u}_{h,2}{\rm d}s
  + \beta_3 \int_\Omega {\rm rot}~\bb{u}_h {\rm d}x & = \langle \bb{f}_h, \bb{v}_h \rangle, \qquad \bb{v}_h\in V_h, \\
\mu_1\int_{\partial \Omega} \bb{v}_{h,1}{\rm d}s & = 0, \qquad \mu_1 \in \mathbb{R},\\
\mu_2\int_{\partial \Omega} \bb{v}_{h,2}{\rm d}s & = 0, \qquad \mu_2 \in \mathbb{R},\\
\mu_3\int_\Omega {\rm rot}_h~\bb{v}_h {\rm d}x & = 0, \qquad \mu_3 \in \mathbb{R},\\
\end{cases}
\end{equation}
where $\bb{u}_h = (\bb{u}_{h,1}, \bb{u}_{h,2})^\intercal$ and $\bb{v}_h = (\bb{v}_{h,1}, \bb{v}_{h,2})^\intercal$.

Let $\bb{\varphi}_i$, $i=1,\cdots, N$ be the nodal basis function of $\bb{V}_h$, where $N$ is the dimension of $\bb{V}_h$. Then we can write
\[\bb{u}_h = \sum\limits_{i=1}^N \chi_i(\bb{u})\bb{\varphi}_i =: \bb{\varphi}^\intercal \bb{\chi}(\bb{u}).\]
Plug the above equation in \eqref{augVEM}, and take $\bb{v}_h = \bb{\varphi}_j$. We have
\[
\begin{cases}
\sum\limits_{i=1}^Na_h(\bb{\varphi}_i,\bb{\varphi}_j)\chi_i
 + \beta_1\int_{\partial \Omega} \bb{\varphi}_{i,1}{\rm d}s
  + \beta_2\int_{\partial \Omega} \bb{\varphi}_{i,2}{\rm d}s
  + \beta_3 \int_\Omega {\rm rot}~\bb{\varphi}_i {\rm d}x & = 0, \quad j=1,\cdots,N, \\
\sum\limits_{i=1}^N \int_{\partial \Omega} \bb{\varphi}_{i,1}{\rm d}s \chi_i & = 0, \\
\sum\limits_{i=1}^N \int_{\partial \Omega} \bb{\varphi}_{i,2}{\rm d}s \chi_i & = 0, \\
\sum\limits_{i=1}^N \int_\Omega {\rm rot}_h~\bb{\varphi}_i{\rm d}x \chi_i & = 0 .
\end{cases}
\]
Let
\[\bb{A} = \Big(a_h(\bb{\varphi}_j,\bb{\varphi}_i)\Big)_{N\times N}, \quad
\bb{d}_1 = \Big(\int_{\partial \Omega} \bb{\varphi}_{i,1}{\rm d}s\Big)_{N\times 1}, \quad
\bb{d}_2 = \Big(\int_{\partial \Omega} \bb{\varphi}_{i,2}{\rm d}s\Big)_{N\times 1}, \quad
\bb{d}_3 = \Big(\int_\Omega {\rm rot}_h~\bb{\varphi}_i{\rm d}x\Big)_{N\times 1}
\]
and
\[\bb{f} = \Big(  \langle \bb{f}_h, \bb{\varphi}_i \rangle \Big)_{N\times 1}.\]
The linear system can be written in matrix form:
\[
\begin{bmatrix}
\bb{A}              & \bb{d}_1 & \bb{d}_2  & \bb{d}_3 \\
\bb{d}_1^\intercal  &     0    &           & \\
\bb{d}_2^\intercal  &          &     0     & \\
\bb{d}_3^\intercal  &          &           & 0\\
\end{bmatrix}
\begin{bmatrix}
\bb{\chi} \\
\beta_1 \\
\beta_2 \\
\beta_3
\end{bmatrix}
=\begin{bmatrix}
\bb{f} \\
0 \\
0 \\
0
\end{bmatrix}.
\]

\subsection{The pure traction problem in the original space}

We first solve the pure traction problem on two different kinds of meshes. One is the uniform triangulation and the other is the unstructured polygonal mesh as shown in Fig.~\ref{fig:mesh}.

\begin{example}\label{example1}
The right-hand side $\bb{f}$ and the boundary conditions are chosen in such a way that the exact solution is
  \[\bb{u}(x,y) =  \begin{bmatrix}
  ( - 1 + \cos 2\pi x)\sin2\pi y \\
   - ( - 1 + \cos 2\pi y)\sin2\pi x
\end{bmatrix}  + \frac{1}{1 + \lambda}\sin \pi x\sin \pi y \begin{bmatrix}
  1 \\
  1
\end{bmatrix}.\]
\end{example}

Note that for the pure traction problem the solutions are unique only up to an additive function in $RM(\Omega)$, which affects the evaluation of the errors in the $L^2$ norm. For this reason, we replace the exact solution by $\bb{u}-\bb{p}$ so that the constraints in \eqref{Htilde} are satisfied, where $\bb{p} \in RM(\Omega)$ is given by
\[\bb{p} = c_0 \begin{bmatrix} 1 \\ 0 \end{bmatrix} + c_1 \begin{bmatrix} 0 \\ 1 \end{bmatrix}
+ c_2 \begin{bmatrix} -y \\ x \end{bmatrix}\]
with
\begin{align*}
&c_2 = \frac{1}{2}\int_\Omega {\rm rot}~\bb{u} {\rm d}x, \\
&c_1 = \frac{1}{|\partial \Omega|} \Big( \int_{\partial \Omega} u_2 {\rm d}s - c_2 \int_{\partial \Omega} x {\rm d}s \Big), \\
&c_0 = \frac{1}{|\partial \Omega|} \Big( \int_{\partial \Omega} u_1 {\rm d}s + c_2 \int_{\partial \Omega} y {\rm d}s \Big).
\end{align*}

\begin{figure}[!htb]
  \centering
  %\subfigure[Polygonal mesh]{\includegraphics[scale=0.5,trim=80 80 80 80,clip]{images/Ex1_PolyFig-eps-converted-to.pdf}}\\
%  \subfigure[Triangular mesh]{\includegraphics[scale=0.5,trim=80 80 80 80,clip]{images/Ex1_TriFig-eps-converted-to.pdf}}\\
  \subfigure[Polygonal mesh]{\includegraphics[scale=0.5]{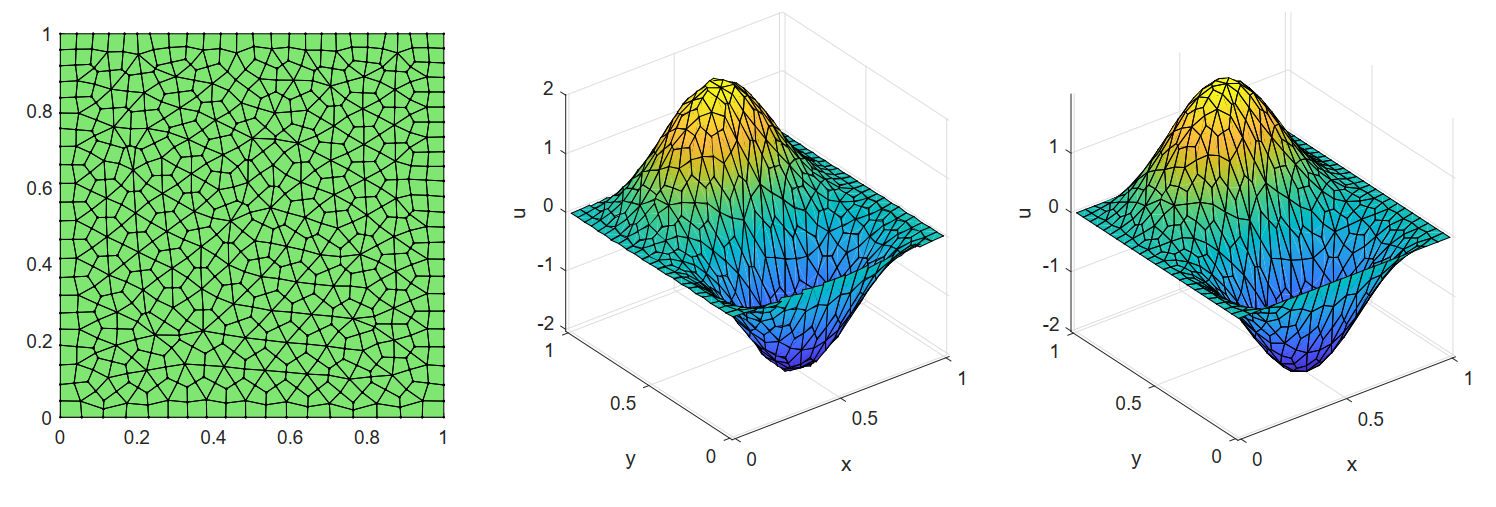}}\\
  \subfigure[Triangular mesh]{\includegraphics[scale=0.5]{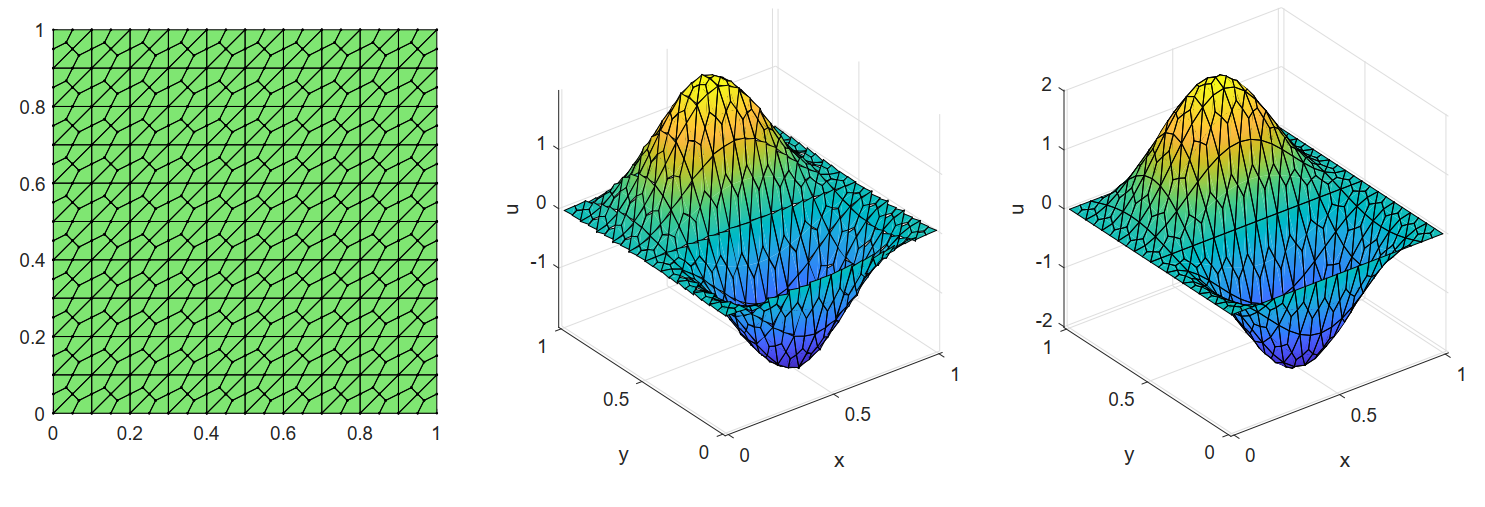}}\\
  \caption{Numerical and exact solutions of Example \ref{example1} with refineType=1. }\label{fig:Ex1_PolyTri}
\end{figure}

Let $\bb{u}$ be the exact solution of \eqref{model} and $\bb{u}_h$ the discrete solution of the proposed VEM \eqref{VEM}.
Since the VEM solution $\bb{u}_h$ is not explicitly known inside the polygonal elements, as in \cite{Beirao-Lovadina-Russo-2017} we will evaluate the errors by comparing the exact solution $\bb{u}$ with the elliptic projection $\Pi _1^E \bb{u}_h$. In this way, the discrete $H^1$ and $L^2$ errors are quantified by
\begin{equation*}
{\rm ErrH1} = \left( \sum\limits_{E \in \mathcal{T}_h^*} | \bb{u} - \Pi_1^E \bb{u}_h |_{1,E}^2  \right)^{1/2} \quad \mbox{and} \quad {\rm ErrL2} = \left( \sum\limits_{E \in \mathcal{T}_h^*} \| \bb{u} - \Pi_1^E \bb{u}_h \|_{0,E}^2 \right)^{1/2},
\end{equation*}
respectively.

To test the accuracy of the proposed method we first consider a sequence of polygonal meshes, which is a Centroidal Voronoi Tessellation of the unit square in 32, 64, 128, 256 and 512 polygons. These meshes are generated by the MATLAB toolbox - PolyMesher introduced in \cite{Talischi-Paulino-Pereira-2012}. For $\lambda=10^{10}$ and $\mu = 1$, we report the nodal values of the elliptic projection $\Pi_1^E \bb{u}_{h,1}$ in Fig.~\ref{fig:Ex1_PolyTri}~(a), where the first type of mesh refinement in Fig.~\ref{fig:refineType}~(a) is used. The convergence order of the errors against the mesh size $h$ is shown in Fig.~\ref{fig:Ex1_PolyTriRat}~(a). We also report the results for triangulation in Fig.~\ref{fig:Ex1_PolyTri}~(b) and Fig.~\ref{fig:Ex1_PolyTriRat}~(b). Generally speaking, $h$ is proportional to $N^{-1/2}$, where $N$ is the total number of elements in the mesh. For each fixed $\lambda$ and $\mu$ the convergence rate with respect to $h$ is estimated by assuming ${\rm Err}(h) = ch^{\alpha}$, and by computing a least squares fit to this log-linear relation. As observed in Fig.~\ref{fig:Ex1_PolyTriRat}, the convergence rate is linear with respect to the $H^1$ norm, and the VEM ensures the quadratic convergence for the $L^2$ norm for different values of $\lambda$, which is consistent with the theoretical prediction in Theorem \ref{thm:err}.

\begin{figure}[!htb]
  \centering
  \subfigure[Polygonal mesh]{\includegraphics[height=5cm]{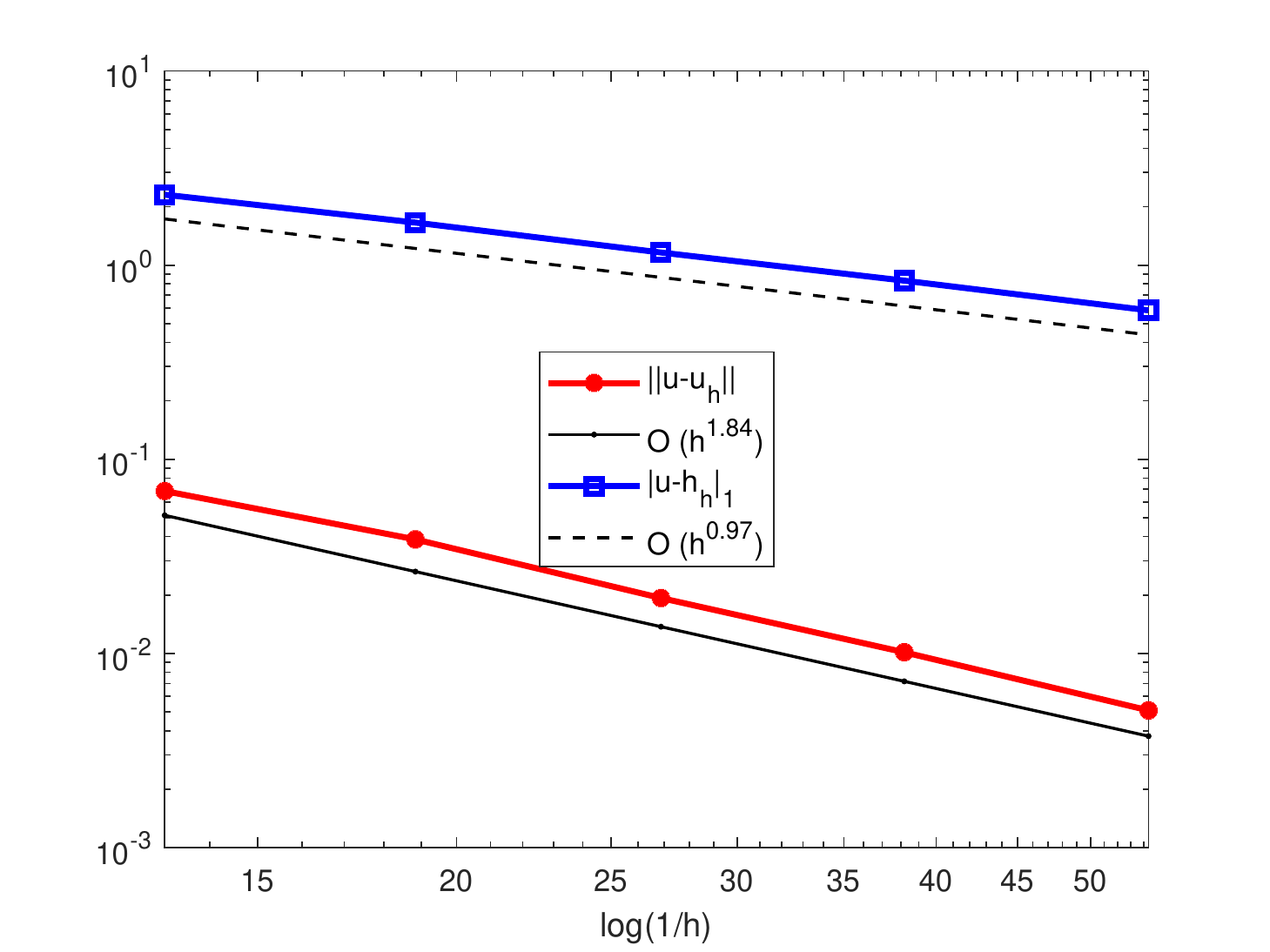}}
  \subfigure[Triangular mesh]{\includegraphics[height=5cm]{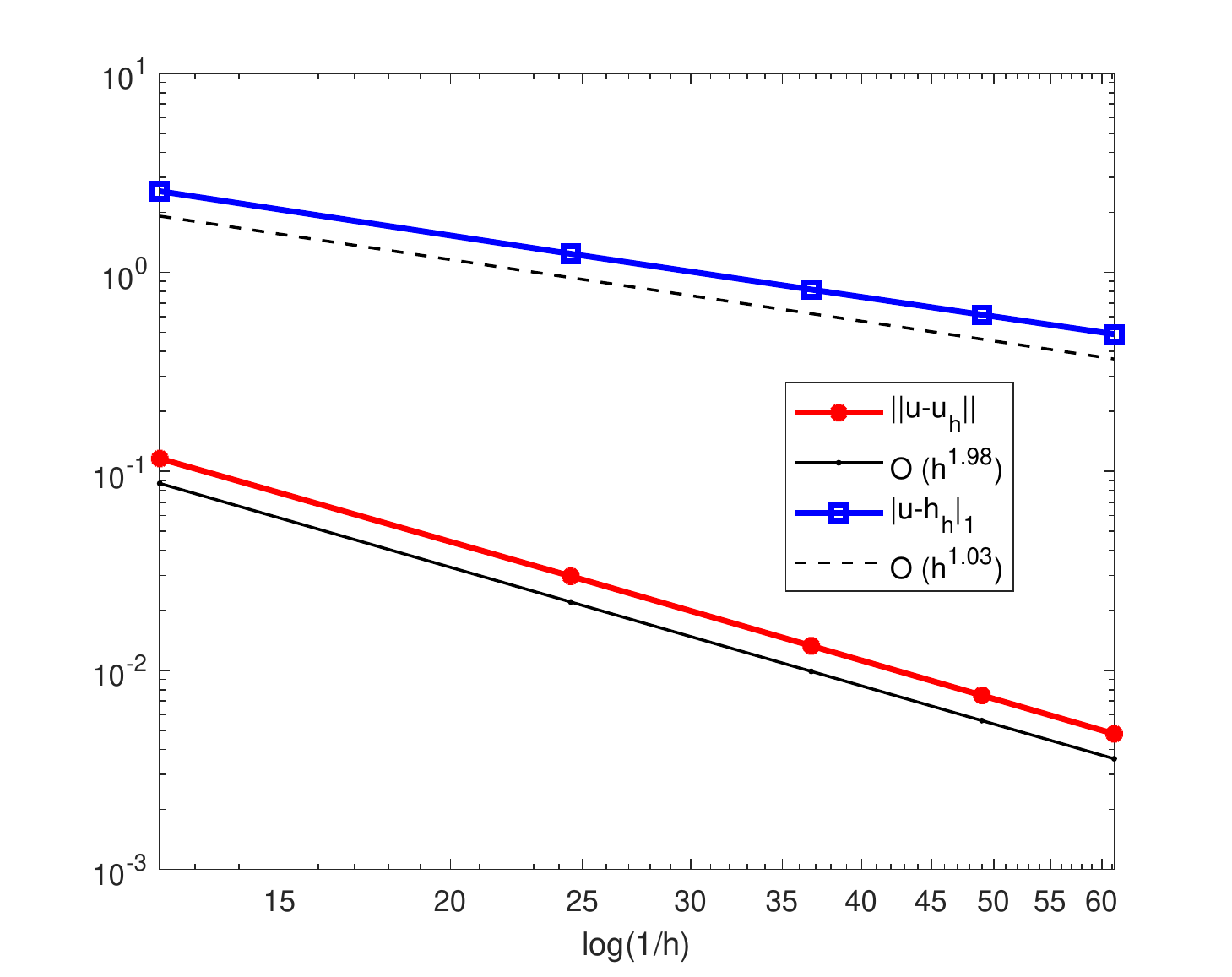}}\\
  \caption{The convergence rate of Example \ref{example1} with refineType=1}\label{fig:Ex1_PolyTriRat}
\end{figure}

We now consider the last two types of mesh refinement in Fig.~\ref{fig:refineType} and only report the numerical results for triangular meshes. As observed in Tab.~\ref{tab:Ex1_Tri2} and Tab.~\ref{tab:Ex1_Tri3}, the optimal rates of convergence are achieved for both subdivisions in the nearly incompressible case.

\begin{table}[!htb]
  \centering
  \caption{Convergence rate of Example \ref{example1} w.r.t. $L^2$ norm for triangulation with refineType=2}\label{tab:Ex1_Tri2}
  \begin{tabular}{ccccccccccccccccc}
  \toprule[0.2mm]
  $\lambda \backslash h$ & $1/5$   &$1/10$   &$1/15$   &$1/20$   &$1/25$  & Rate\\
  \midrule[0.3mm]
   $10^0$     & 1.1208e-01  &2.9679e-02  &1.3421e-02  &7.6091e-03 &4.8911e-03 & 1.95\\
   $10^2$     & 1.0645e-01  &2.8443e-02  &1.2860e-02  &7.2846e-03 &4.6791e-03 & 1.94\\
   $10^4$     & 1.0640e-01  &2.8434e-02  &1.2856e-02  &7.2822e-03 &4.6775e-03 & 1.94\\
   $10^6$     & 1.0640e-01  &2.8434e-02  &1.2855e-02  &7.2821e-03 &4.6775e-03 & 1.94\\
   $10^8$     & 1.0640e-01  &2.8434e-02  &1.2855e-02  &7.2820e-03 &4.6777e-03 & 1.94\\
  \bottomrule[0.2mm]
\end{tabular}
\end{table}

\begin{table}[!htb]
  \centering
  \caption{Convergence rate of Example \ref{example1} w.r.t. $L^2$ norm for triangulation with refineType=3}\label{tab:Ex1_Tri3}
  \begin{tabular}{ccccccccccccccccc}
  \toprule[0.2mm]
  $\lambda \backslash h$ & $1/5$   &$1/10$   &$1/15$   &$1/20$   &$1/25$  & Rate\\
  \midrule[0.3mm]
   $10^0$     & 3.0839e-01  &7.6958e-02   &3.4534e-02   &1.9499e-02   &1.2500e-02 & 1.99\\
   $10^2$     & 2.9286e-01  &7.2817e-02   &3.2637e-02   &1.8411e-02   &1.1795e-02 & 1.99\\
   $10^4$     & 2.9265e-01  &7.2761e-02   &3.2611e-02   &1.8396e-02   &1.1785e-02 & 1.99\\
   $10^6$     & 2.9265e-01  &7.2760e-02   &3.2611e-02   &1.8396e-02   &1.1785e-02 & 1.99\\
   $10^8$     & 2.9265e-01  &7.2760e-02   &3.2611e-02   &1.8396e-02   &1.1785e-02 & 1.99\\
  \bottomrule[0.2mm]
\end{tabular}
\end{table}

\begin{example}\label{divfree}
We further investigate the impact of Lam\'{e} constant $\lambda$ to solve a completely
incompressible problem. The exact solution is given by
 \[u_1 = - \sin^3(\pi x) \sin(2\pi y) \sin(\pi y), \quad u_2 = \sin(2\pi x) \sin(\pi x) \sin^3(\pi y).\]
 One can check that ${\rm div}~\bb{u} = 0$ and $\bb{f}$ is independent of $\lambda$.
\end{example}

We test the performance of our VEMs for polygonal meshes and triangulations with $\lambda = 10^{10}$ and $\mu=1$ fixed. The error orders are shown in Fig.~\ref{fig:divfree}, from which we observe that the optimal rates of convergence are achieved for both subdivisions in the completely incompressible limiting problem.

\begin{figure}[!htb]
  \centering
  \subfigure[Polygonal mesh]{\includegraphics[height=5cm]{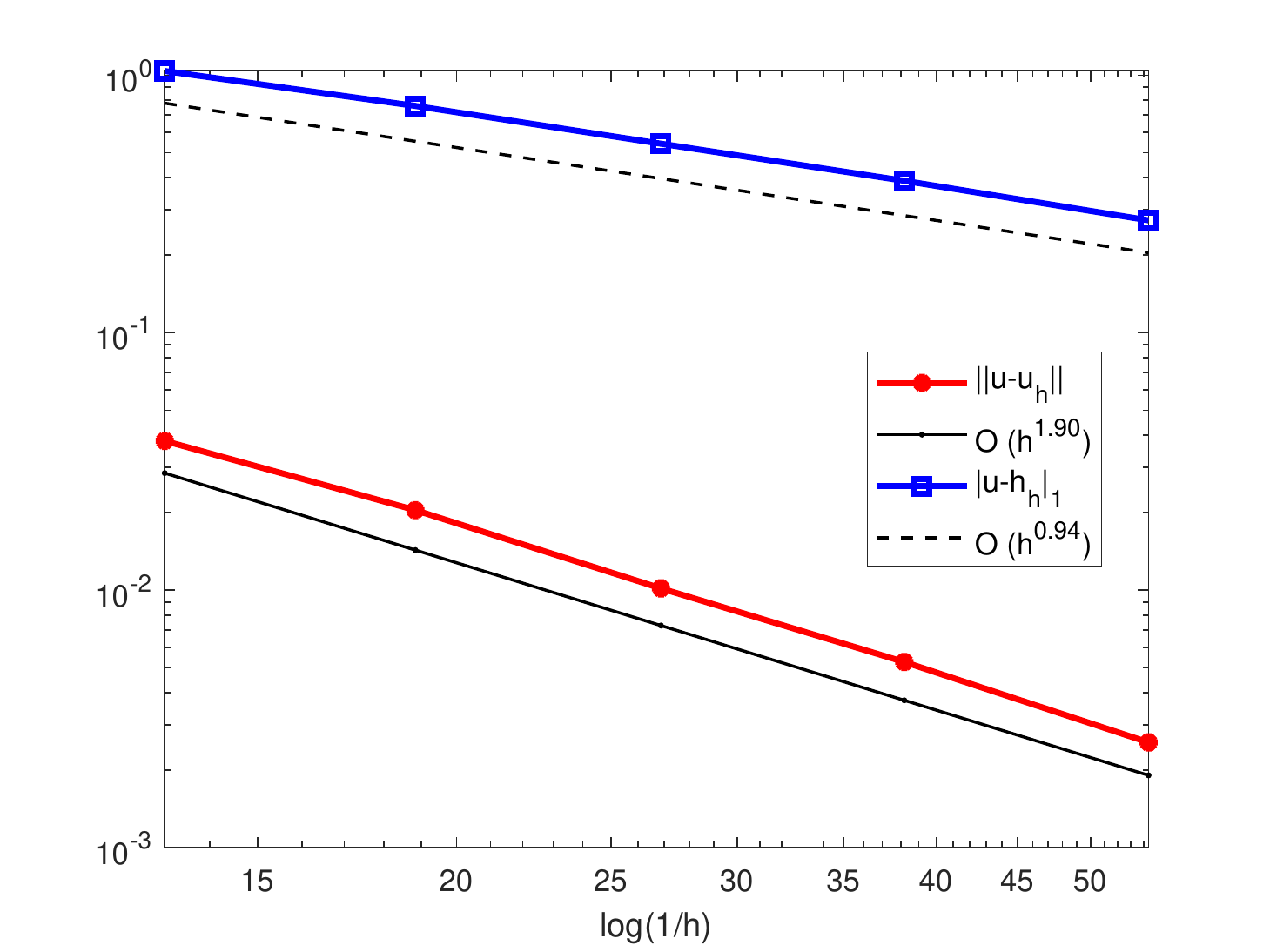}}
  \subfigure[Triangular mesh]{\includegraphics[height=5cm]{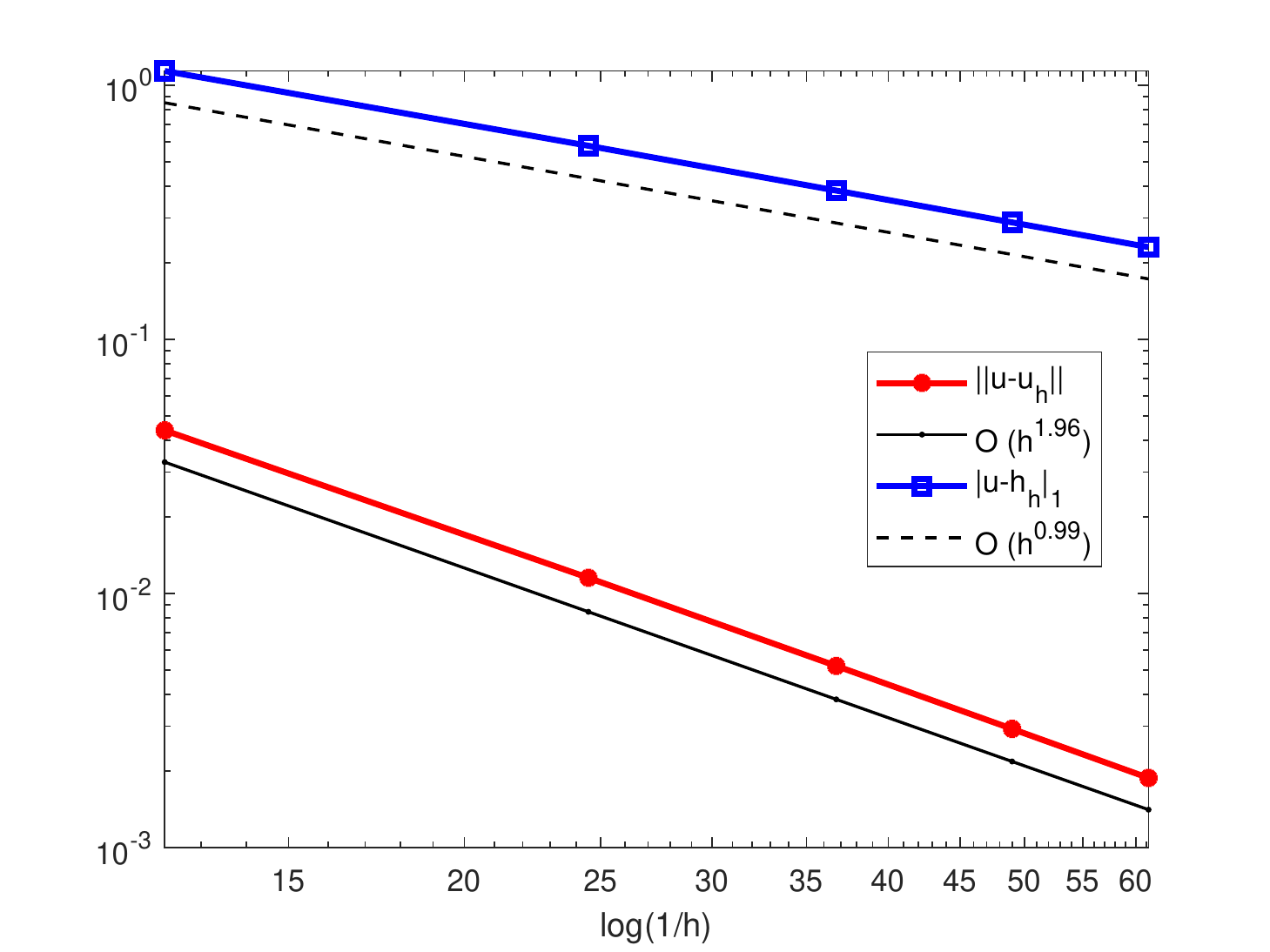}}\\
  \caption{The convergence rate of Example \ref{divfree} with refineType=1}\label{fig:divfree}
\end{figure}

\subsection{The pure traction problem in the enhancement space}

\begin{figure}[!htb]
  \centering
  %\subfigure[Solutions]{\includegraphics[scale=0.5,trim=80 80 80 80,clip]{images/Ex4_distort3-eps-converted-to.pdf}}\\
  \subfigure[Solutions]{\includegraphics[scale=0.5,trim=80 80 80 80,clip]{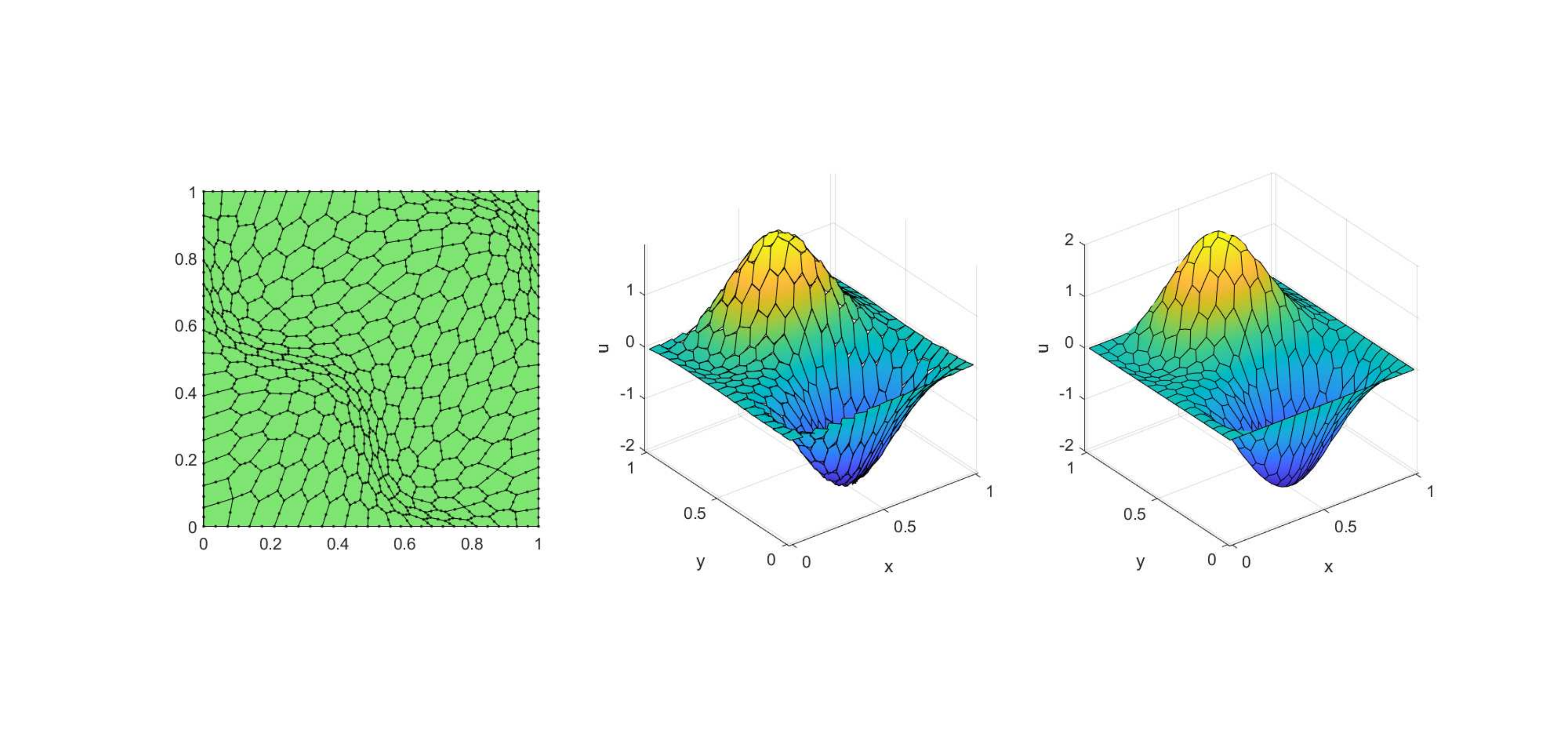}}\\
  \subfigure[Convergence rate]{\includegraphics[scale=0.5]{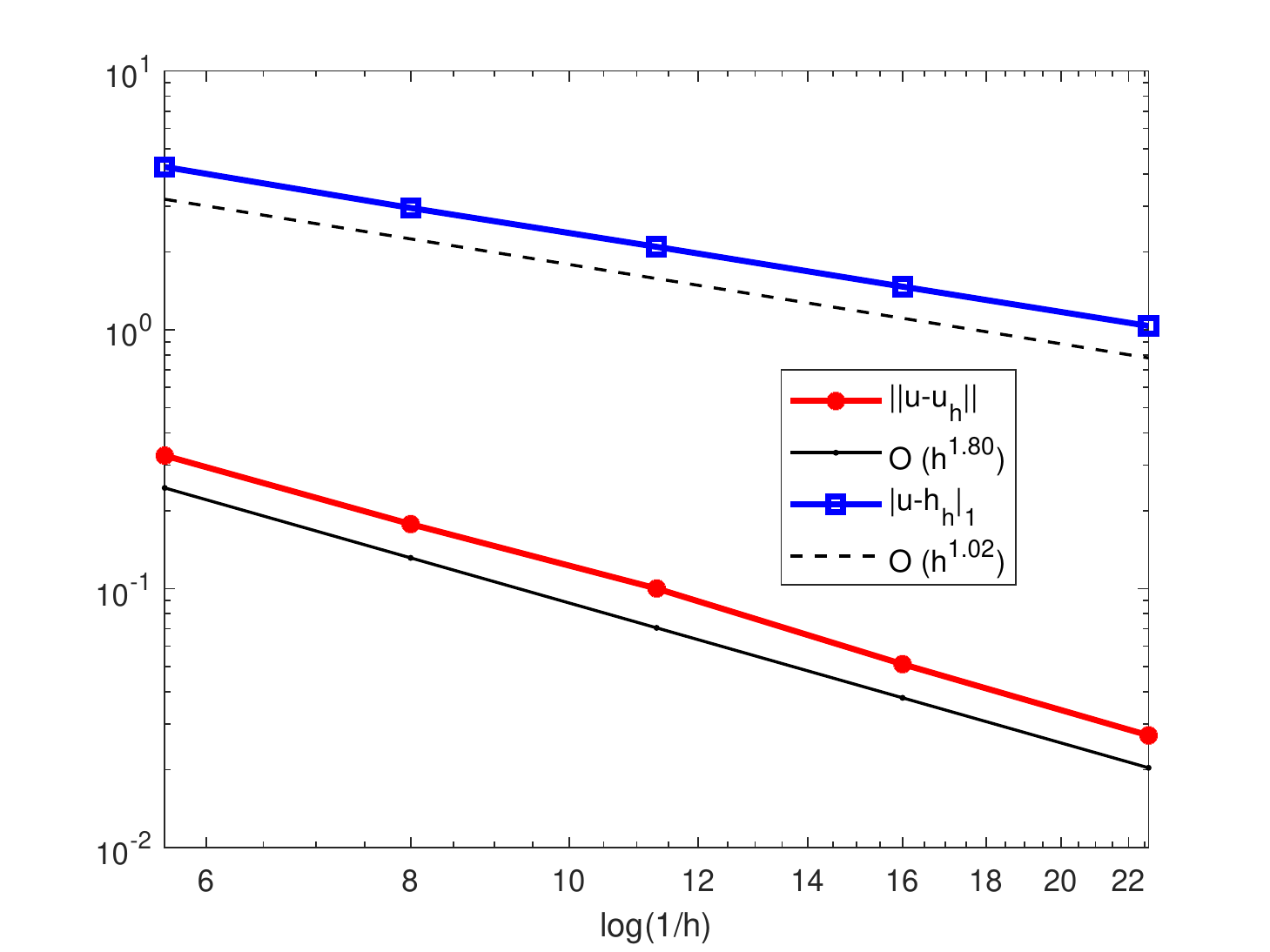}}\\
  \caption{Numerical result of the VEM in the enhancement space \eqref{enhancement} for distorted polygonal meshes with refineType=3. }\label{fig:distort4}
\end{figure}

We can also test the performance for the VEM in the enhancement space \eqref{enhancement}. In the implementation of \eqref{augVEM}, one just needs to replace the integrands on $\partial \Omega$ with the ones on $\Omega$. We still consider Example \ref{example1} and present the numerical result in Fig.~\ref{fig:distort4}. Similar behaviours are observed, which confirms the discussion in Section \ref{sec:discussion}.

\subsection{The problem with mixed boundary conditions}

\begin{example}\label{ex:mixed}
The proposed virtual element method also applies to the pure displacement problem or the problem with mixed boundary conditions:
\[
\begin{cases}
- {\rm div}~\bb{\sigma}(\bb{u}) = \bb{f}  \quad & \mbox{in} ~~\Omega, \\
\bb{\sigma}(\bb{u})\bb{n} = \bb{g}_1 \quad & \mbox{on} ~~\Gamma, \\
\bb{u}  = \bb{g}_2  \quad & \mbox{on} ~~\partial\Omega \backslash \Gamma.\\
\end{cases}
\]
\end{example}

We still consider the exact solution in Example \ref{divfree} with Dirichlet boundary condition imposed on $y=0$ and
repeat the numerical simulation for the distorted mesh as shown in Fig.~\ref{fig:distortionSquare}. We remark that in this case there is no need to introduce the Lagrange multipliers $\beta_i$.

\begin{figure}[!htb]
  \centering
  \includegraphics[scale=0.5]{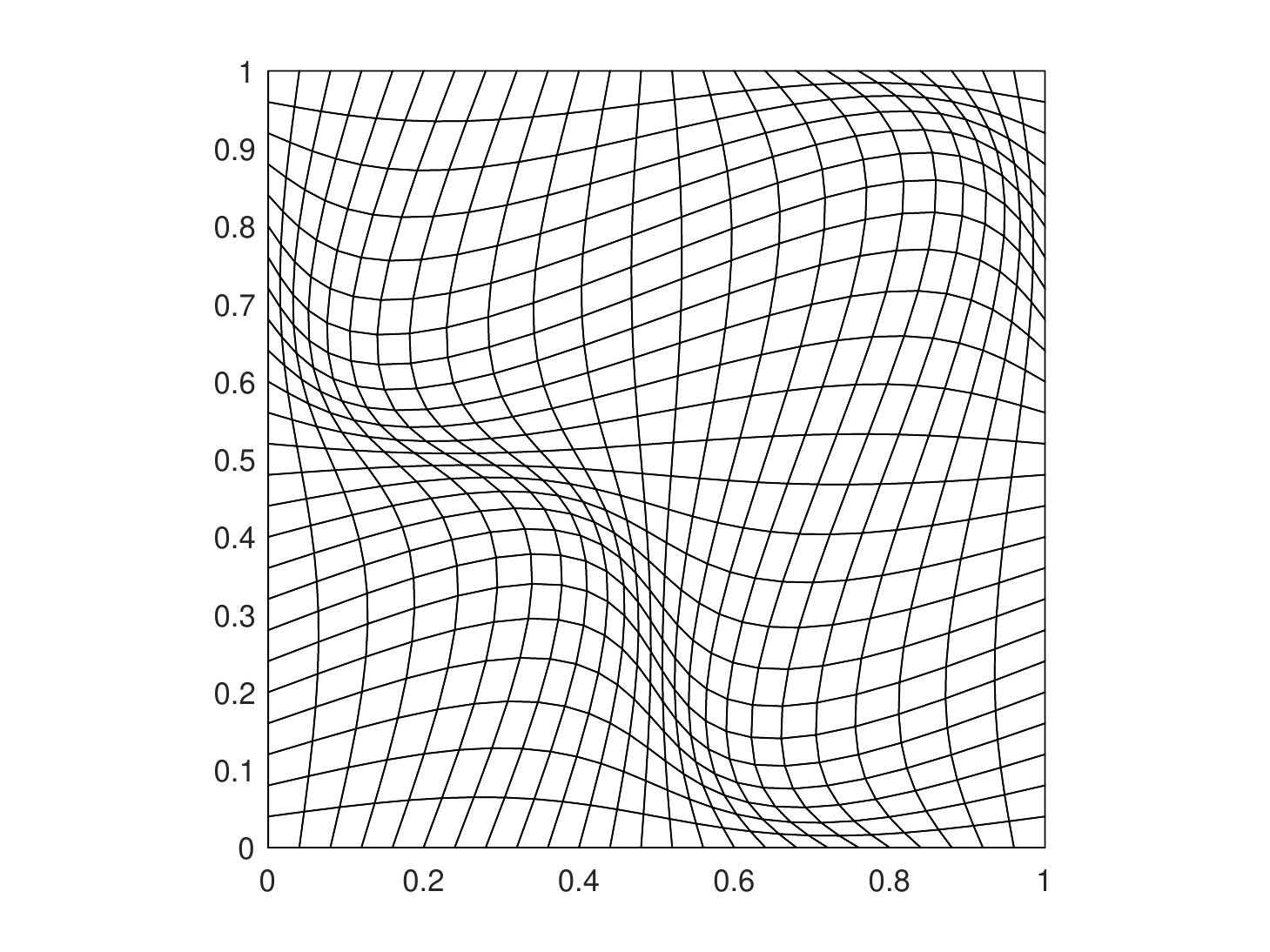}\\
  \caption{A distorted rectangular mesh}\label{fig:distortionSquare}
\end{figure}

Let $(\xi, \eta)$ be the coordinates on the original mesh. The nodes of the distorted mesh are obtained by the following transformation
\[
x = \xi + t_c \sin(2\pi\xi)\sin(2\pi\eta), \quad y = \eta + t_c \sin(2\pi\xi)\sin(2\pi\eta),
\]
where $(x,y)$ is the coordinate of new nodal points; $t_c$, taken as $0.1$ in the computation, is the distortion parameter.

\begin{figure}[!htb]
  \centering
  %\subfigure[Rectangular mesh]{\includegraphics[scale=0.5,trim=80 80 80 80,clip]{images/Ex3_distort1-eps-converted-to.pdf}}\\
  %\subfigure[Polygonal mesh]{\includegraphics[scale=0.5,trim=80 80 80 80,clip]{images/Ex3_distort1Poly-eps-converted-to.pdf}}\\
  \subfigure[Rectangular mesh]{\includegraphics[scale=0.5]{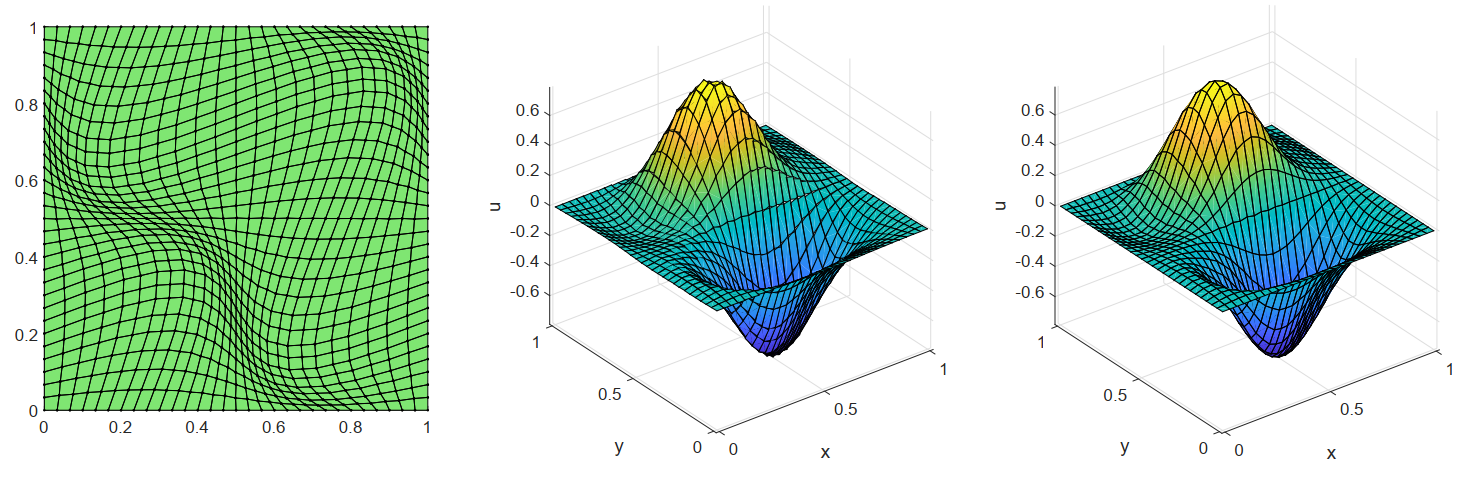}}\\
  \subfigure[Polygonal mesh]{\includegraphics[scale=0.5]{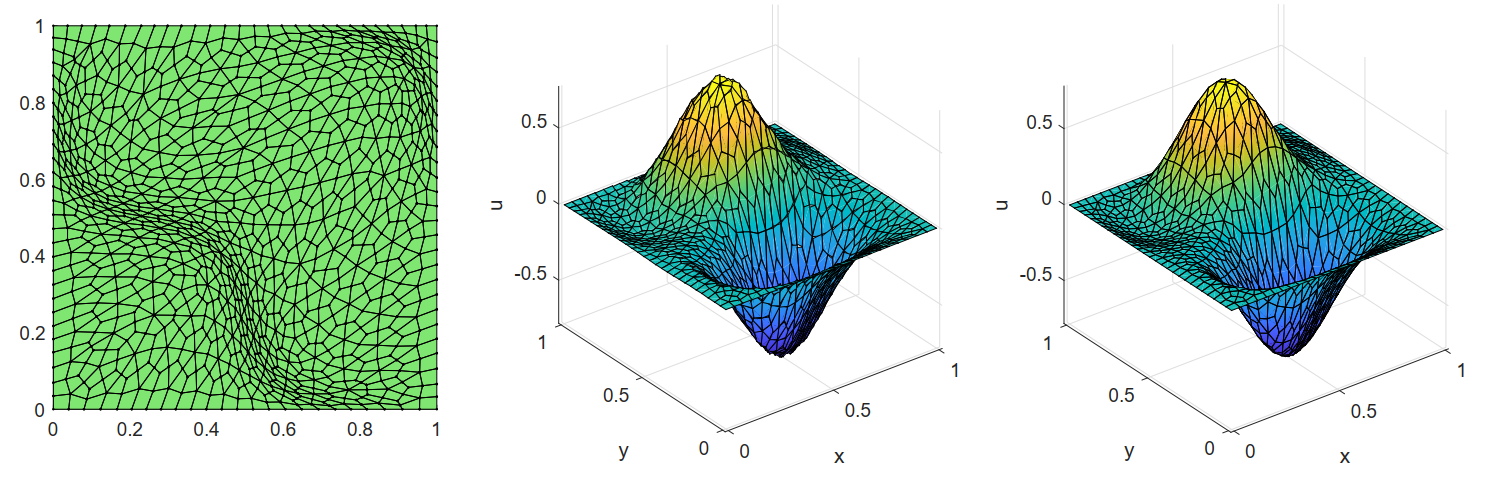}}\\
  \caption{Numerical and exact solutions of Example \ref{ex:mixed} for distorted mesh with refineType=1. }\label{fig:distort}
\end{figure}

\begin{table}[!htb]
  \centering
  \caption{Convergence rate of Example \ref{ex:mixed} w.r.t. $L^2$ norm for quadrilateral distorted mesh with refineType=2}\label{tab:Ex3rat2}
  \begin{tabular}{ccccccccccccccccc}
  \toprule[0.2mm]
  $\lambda \backslash h$ & $1/5$   &$1/10$   &$1/15$   &$1/20$   &$1/25$  & Rate\\
  \midrule[0.3mm]
   $10^2$     & 6.7057e-02   &2.1470e-02   &9.8523e-03   &5.6095e-03   &3.6112e-03 & 1.82\\
   $10^8$     & 6.7205e-02   &2.1516e-02   &9.8757e-03   &5.6234e-03   &3.6204e-03 & 1.82\\
  \bottomrule[0.2mm]
\end{tabular}
\end{table}

Fig.~\ref{fig:distort}~(a) and (b) display the simulation results for the first type of mesh refinement on quadrilateral distorted mesh and general polygonal distorted mesh, respectively. The numerical solutions are well matched with the exact ones as observed. The errors in the $L^2$ norm for the second and third types of mesh refinement are listed in Tab.~\ref{tab:Ex3rat2} and Tab.~\ref{tab:Ex3rat3}, respectively. Despite the very distorted meshes, we can still obtain almost second-order convergence for different values of $\lambda$.

\begin{table}[!htb]
  \centering
  \caption{Convergence rate of Example \ref{ex:mixed} w.r.t. $L^2$ norm for quadrilateral distorted mesh with refineType=3}\label{tab:Ex3rat3}
  \begin{tabular}{ccccccccccccccccc}
  \toprule[0.2mm]
  $\lambda \backslash h$ & $1/5$   &$1/10$   &$1/15$   &$1/20$   &$1/25$  & Rate\\
  \midrule[0.3mm]
   $10^2$     & 1.5715e-01   &5.0710e-02   &2.3895e-02   &1.3758e-02   &8.9081e-03 & 1.78\\
   $10^8$     & 1.5730e-01   &5.0792e-02   &2.3943e-02   &1.3789e-02   &8.9297e-03 & 1.78\\
  \bottomrule[0.2mm]
\end{tabular}
\end{table}

\begin{table}[!htb]
  \centering
  \caption{Convergence rate of Example \ref{example1} w.r.t. $L^2$ norm for the unified locking-free scheme (Polygonal meshes with refineType=1)}\label{tab:uniform}
  \begin{tabular}{ccccccccccccccccc}
  \toprule[0.2mm]
  $\lambda \backslash h$ & $1/5$   &$1/10$   &$1/15$   &$1/20$   &$1/25$  & Rate\\
  \midrule[0.3mm]
   $10^0$     & 6.7235e-02   &3.6869e-02   &1.8102e-02   &9.4467e-03   &4.7476e-03 & 1.88\\
   $10^2$     & 6.9307e-02   &3.8020e-02   &1.8849e-02   &9.8457e-03   &4.9499e-03 & 1.87\\
   $10^4$     & 6.9387e-02   &3.8072e-02   &1.8875e-02   &9.8609e-03   &4.9582e-03 & 1.87\\
   $10^6$     & 6.9387e-02   &3.8072e-02   &1.8875e-02   &9.8611e-03   &4.9583e-03 & 1.87\\
   $10^8$     & 6.9387e-02   &3.8072e-02   &1.8875e-02   &9.8611e-03   &4.9583e-03 & 1.87\\
  \bottomrule[0.2mm]
\end{tabular}
\end{table}

\subsection{The unified locking-free scheme}

For the unified locking-free scheme, the subroutine {\color{gray} elasticityVEM\_NCreducedIntegration.m} can be modified with slight changes.  The new M-function is {\color{gray} elasticityVEM\_NCUniformReducedIntegration.m} and the test script {\color{gray} main\_elasticityVEM\_NCUniformReducedIntegration.m} verifies the convergence rates.
We still consider Example \ref{example1} with numerical results presented in Tab.~\ref{tab:uniform}, from which we observe a similar behaviour as that of the original VEMs.

\section{Conclusions}

We proposed a locking-free nonconforming VEM in the lowest order case for solving the linear elastic problems, which is an evolution of the locking-free finite element in \cite{Falk-1991} to general polygonal meshes by using the reduced integration technique. The parameter-free property and the optimal error estimate are established. A unified scheme both for the conforming and nonconforming VEMs is also constructed. Numerical results are consistent with theoretical findings. The proposed method can be adapted to a mixed formulation, as was done in \cite{Falk-1991}.

%\section*{Acknowledgements}

\end{document}